\RequirePackage{fix-cm}
\documentclass[envcountsect,numbook,natbib]{svjour3_arxiv4}%[envcountsect,numbook]{svjour3_arxiv}   
\smartqed  % flush right qed marks, e.g. at end of proof
\usepackage{graphicx}
\graphicspath{{./figures/}}
\usepackage[utf8]{inputenc}
\usepackage{dsfont}
\usepackage{subfigure}
\usepackage{color}
\usepackage{url}
\usepackage{enumerate} %newly added package
\usepackage{booktabs}
\usepackage{tabularx}
\usepackage{ltablex}
\usepackage{mathptmx}
\usepackage{amssymb,amsmath,amsfonts,latexsym} 
\usepackage{longtable} 
\usepackage{textcomp}
\usepackage{lscape}
\usepackage[gen]{eurosym}  % For euro symbol
\usepackage{bm}
\usepackage{bbm}
\usepackage{multirow}
\usepackage{mathtools}
\mathtoolsset{showonlyrefs=true}
\usepackage[pdftex]{hyperref}
\usepackage{calrsfs}  % To set a new command

\usepackage{appendix}

\usepackage[12pt]{extsizes}
\usepackage{geometry}
\usepackage{ifthen}
\allowdisplaybreaks

\spnewtheorem{theorem}{Theorem}{\bfseries}{\itshape}
\spnewtheorem{corollary}[theorem]{Corollary}{\bfseries}{\itshape}
\spnewtheorem{lemma}[theorem]{Lemma}{\bfseries}{\itshape}
\spnewtheorem{proposition}[theorem]{Proposition}{\bfseries}{\itshape}
\spnewtheorem{definition}[theorem]{Definition}{\bfseries}{\itshape}
\spnewtheorem{remark}[theorem]{Remark}{\bfseries}{\upshape}
\spnewtheorem{assumption}[theorem]{Assumption}{\bfseries}{\itshape}
\counterwithin{theorem}{section}
\smartqed

\renewcommand{\paragraph}[1]{{\bf #1.}}
\definecolor{myred}{rgb}{0.8,0,0}  %rot

{\vskip\baselineskip\noindent\textbf{Proof of {#1}:}}%
{\hspace*{.1pt}\hspace*{\fill}\BOX\vskip\baselineskip}

\usepackage{tikz}

\newcommand{\diffns}{\mathrm{d}}

\DeclareMathAlphabet{\pazocal}{OMS}{zplm}{m}{n}
\let\mathcal\pazocal
\begin{document}
	
	\title{A stochastic maximum principle for singular mean-field regime-switching optimal control  \\  
	}

	\titlerunning{A stochastic maximum principle for singular mean-field regime-switching optimal control}        
	
	\author{Maalvlad\'edon Ganet Som\'e \and Edward Korveh}

	\authorrunning{M. Ganet Som\'e, E. Korveh} % if too long for running head

	\institute{Maalvlad\'edon Ganet Som\'e \at
		Department of Mathematics, School of Science, College of Science and Technology, University of Rwanda, Kigali 4285, Rwanda \\
		African Institute for Mathematical Sciences (AIMS), Ghana, P.O. Box LGDTD 20046, Legon, Accra,  Shoppers Street, Spintex, Accra, Ghana;  
		%              Tel.: +49 151 42844783\\    
		\email{\texttt{maalvladedon@aims.edu.gh}}
		\and
		Edward Korveh \at
		University of Ghana, Department of Mathematics, P.O. Box LG 25, Legon,  Ghana; \\
		African Institute for Mathematical Sciences (AIMS), Ghana, P.O. Box LGDTD 20046, Legon, Accra,  Shoppers Street, Spintex, Accra, Ghana;  
		%              Tel.: +49 151 42844783\\    
		\email{\texttt{ekorveh@ug.edu.gh}}           
	}
	
	\date{Version of  \today}

	\maketitle
	
	\begin{abstract}
	In this paper, we investigate a mean-field singular stochastic optimal control problem for systems governed by mean-field regime-switching singular stochastic differential equations. The state process is assumed to depend on both a regular and a singular control, and the coefficient associated with the singular component is allowed to be regime dependent. We derive both necessary and sufficient singular stochastic maximum principles. Because the regular control domain is not assumed to be convex, we employ the spike variation technique and obtain the necessary maximum principle by introducing a second-order adjoint process. As an application, we use the main theoretical results to analyse an inter-bank borrowing and lending model with transaction costs.
		
	\end{abstract}

	\section{Introduction}
	%The Introduction should contain a summary of the literature that pertains to the paper and a summary of the contents of the paper, as much as possible without formulas. It should spell out what is the novel contribution of the paper.  In the section References, the references should be listed in alphabetical order. Then, references should be cited in the text by the reference number in square brackets, for example, \cite{Huang} or \cite{Huang,Miele} or 
	%\cite{Huang,Miele,Nocedal,Ralston,Yang}.
	
	\label{sec:man6intro}
	%\subsection*{Background and motivation}
	
	Mean-field type stochastic control has developed rapidly in recent years. These models originated in statistical physics, where the limiting behaviour of large systems of interacting particles is described through their aggregate, or mean-field, dynamics. They are now widely used in engineering, neuroscience, economics, finance, and insurance. In parallel, regime-switching models have become increasingly important in finance, as financial markets often experience abrupt structural changes caused by exogenous shocks. Such regime changes are commonly modelled using finite-state Markov chains (see for example, \cite{zhang2012stochastic} and references therein).
	
	Beyond classical (or \textsl{regular}) controls, singular controls also play a crucial role in many finance-related applications. A singular control is an adapted, non-decreasing, c\`adl`ag, finite-variation process that is singular with respect to Lebesgue measure. These controls capture interventions that do not operate through smooth, continuous adjustment rates; examples include consumption or withdrawal decisions, transaction costs, harvesting, and other lump-type actions.
	
In this paper, the mean-field effect is specified via the marginal distribution of a prescribed function of the state process. %This formulation is more general than many existing models: several standard frameworks can be recovered by an appropriate choice of this function. 
In our setting, the drift, diffusion, and regime-transition characteristics of the state dynamics depend on both this mean-field term and the regular control. The singular control enters the state stochastic differential equation (SDE) as an additional process whose coefficient is allowed to vary across regimes.
	
	Two main approaches exist for solving stochastic optimal control problems: the dynamic programming method and the stochastic maximum principle. In this paper, we adopt the latter approach. Because our objective functional depends on distributional characteristics of the state, the problem is time-inconsistent, and classical dynamic programming arguments are generally not directly applicable. Our formulation permits the objective to depend on the expectation of a function of the state process, rather than solely on the expectation of the state itself, yielding an even more general class of problems.
	
	The stochastic maximum principle was introduced by Kushner \cite{kushner1972necessary, kushner1965stochastic}, and a rich literature has since emerged, including many contributions for pure diffusion models \cite{bahlali1996maximum, bensoussan1981lectures,cadenillas1995stochastic,elliott1990optimal,haussmann1986stochastic}.
	%and in the jump-diffusion case by \cite{framstad2001sufficient,framstad2004sufficient,jingtao2011risk}.
	Mean-field models were introduced into finance by Lasry and Lions \cite{lasry2007mean}after which mean-field type control problems have been studied extensively (see, for example, \cite{andersson2011maximum, buckdahn2009mean, hafayed2014singular,Thilo2012} and references therein). For regime-switching models, maximum principle-type results have been established in  \cite{donnelly2011sufficient, Menou20142, MeMo17,  nguyen2020stochastic,tao2012maximum}. Singular mean-field control problems have been addressed in \cite{dahl2017singular, hu2017singular}, where necessary and sufficient maximum principles were derived for singular SDEs; in those works, the corresponding adjoint equations are mean-field backward SDEs with reflection.
	
	In the present paper, we establish both necessary and sufficient maximum principles for a general singular mean-field control problem under regime switching and multiple sources of uncertainty. Our results strictly extend those in \cite{andersson2011maximum, bahlali2005general, hafayed2014singular}. Furthermore, the conclusions remain valid when the state dynamics include an additional Poisson jump component, thereby generalising the results in \cite{zhang2018general}. This current work is different from that of Wu and Zhang \cite{wu2024maximum} which studied a stochastic control problem for a forward-backward regime-switching system with conditional mean-field interactions using impulse controls. In our work, we considered general mean-field interactions and obained necessary and sufficient principles of optimality via spike variation.
	
	%\textcolor{red}{You did not cite the paper of the guys in JMAA and did not say how your result is different from them. In addition, you did not exactly say what is the difficulty in your work. What novelty are you bringing. It is important to stress that } 
	%\textcolor{red}{This current work is also different from that of \citet{wu2024maximum}, which studied a stochastic control problem for a forward-backward regime-switching system with conditional mean-field interactions using impulse controls. In addition, the coefficient of the impulse control did not depend on the state process. In our work, we considered general mean-field interactions, and if we restrict ourselves to conditional mean-field interactions, we shall recover the results in \cite[Theorem 3.1, and Theorem 4.1]{wu2024maximum}, thus, making the current work more general.}
	
	%	\textcolor{red}{In applying the maximum principle, the associated adjoint process is a solution to a singular regime-switching mean-field BSDE. Existence and uniqueness of solutions to such BSDEs under assumptions in \cite{ParZh98} can be obtained by combining the approaches in \cite{agram2022mean,buckdahn2009mean, ParZh98}. }
	
	%	\textcolor{red}{More precisely, assuming the coefficient are Lipschitz continuous and the finite variation norm of the singular process is small enough, we use Emery type inequalities (see \cite{emery1978stabilite, protter2005stochastic}) to prove in a companion paper that the obtained BSDE has a unique solution.}
	
The remainder of the paper is organised as follows. Section \ref{sec:man6problemformulation} introduces the model, formulates the stochastic optimal control problem, and provides conditions for the existence and uniqueness of solutions. Section \ref{sec:mainresultsandproofs} presents the main results—namely, the necessary and sufficient maximum principles—along with their proofs. Section \ref{sec:application} applies the theoretical findings to an inter-bank borrowing and lending problem with transaction costs. Section \ref{sec:man6conclusion} concludes the paper. Auxiliary lemmas and technical results are given in Appendix \ref{sec:proofofauxi}.

	\section{A Mean-field singular Markov regime switching model and the control problem}
	\label{sec:man6problemformulation}
	Let $T>0$ be fixed finite time horizon, and $(\Omega, \mathcal{F}, \{\mathcal{F}_{t}\}_{t\in [0,T]}, \mathbb{P})$ be a complete filtered probability space. The filtration $\{\mathcal{F}_{t}\}_{t\in [0,T]}$ is right-continuous and $\mathbb{P}$-completed to which all process defined below, including a Markov chain $\alpha := \{\alpha(t)\}_{t\in [0,T]}$, a $1$-dimensional standard Brownian motion $ B:= \left\{B(t)\right\}_{t\in [0,T]}$, and the nondecreasing process $\xi$, are adapted.
	$\alpha$ is a continuous time homogeneous and irreducible Markov chain with the finite state space $\mathcal{S} := \{e_1,e_2,\ldots,e_D\}$, with $D\in \mathbb{N}, e_i\in \mathbb{R}^D$ and for $i,j = 1,2,\ldots,D$, the $j$-th component of $e_i$ is the Kronecker delta $\delta_{ij}$. $\mathbb{N}$ denotes the set of natural numbers. $\mathcal{G} := [\zeta_{ij}]_{i,j=1,2,\ldots,D}$ denotes the generator of the Markov chain $\alpha$ under $\mathbb{P}$ also known as the rate or the $Q$-matrix. For each $i,j = 1,2,\ldots,D$, each entry $\zeta_{ij}$ of the rate matrix  represents the constant transition intensity of the chain from state $e_{i}$ to state $e_{j}$ at time $t$. Without loss of generality, we suppose  $\zeta_{ij}(t) > 0,\,i\neq j$ and $\sum\limits_{j=1}^{D}\zeta_{ij} = 0$, that is $\zeta_{ii}(t) < 0$. The following semimartingale dynamics for the Markov chain holds(see \cite{elliott2008hidden, zhang2012stochastic}):
	\[
	\mathrm{d}\alpha(t) = \mathcal{G}(t)\alpha(t)\mathrm{d}t + \mathrm{d}{\bf{M}}(t), \ \alpha(0)\in\mathbb{R}^{D},\ t\in[0,T],
	\]
	where $\{{\bf{M}}(t) | t\in[0,T]\}$ is an $\mathbb{R}^{D}$-valued, $(\{\mathcal{F}_{t}\}_{t\in[0,T]}, \mathbb{P})$-martingale. %Next, we introduce a set of Markov jump martingales associated with the chain $\alpha$. 
	For each $i,j = 1,2,\ldots, D$, with $i\ne j$ and $t\in [0,T]$, denote by $\mathcal{J}^{ij}(t)$ the number of jumps from state $e_{i}$ to state $e_{j}$ up to time $t$. Then using the martingale dynamics (see \cite{elliott1994exact}) %Then using the martingale dynamics of the chain, one can show (see \cite{elliott1994exact})
	\begin{align*}
		\mathcal{J}^{ij} &:= \sum\limits_{0<s<t}\left\langle \alpha(s-),e_{i} \right\rangle \left\langle \alpha(s),e_{j} \right\rangle= \zeta_{ij}\int_{0}^{t}\left\langle \alpha(s-),e_{i} \right\rangle\mathrm{d}s + m_{ij}(t),
	\end{align*}
	where $m_{ij}(t):= \{m_{ij}(t) | t\in[0,T]\}$ with $m_{ij}(t):= \int_{0}^{t}\left\langle \alpha(s-),e_{i} \right\rangle \left\langle \mathrm{d}{\bf{M}}(s), e_{j} \right\rangle$ is an $(\{\mathcal{F}_{t}\}_{t\in [0,T]},\mathbb{P})$-martingale. The $m_{ij}$'s are also known as basic martingales associated with $\alpha$. For each fixed $j = 1,2,\ldots,D$, let $\Phi_j(t)$ denote the number of jumps into state $e_{j}$ up to time $t$. Then 
	\begin{align*}
		\Phi_j(t) &= \sum\limits_{i=1, i\ne j}^{D}\mathcal{J}^{ij}(t) = \zeta_{j}(t) + \tilde{\Phi}_{j}(t),
	\end{align*}
	where $\zeta_{j}(t) := \sum\limits_{i=1, i\ne j}^{D}\zeta_{ij}\int_{0}^{t}\left\langle \alpha(s), e_{i} \right\rangle\mathrm{d}s$, $\tilde{\Phi}_{j}(t) := \sum\limits_{i=1, i\ne j}^{D}m_{ij}(t)$, and for each $j=1,2,\ldots,D$, $\tilde{\Phi}_{j} := \{\tilde{\Phi}_{j}(t) | t\in [0,T]\}$ is an $(\{\mathcal{F}_{t}\}_{t\in [0,T]}, \mathbb{P})$-martingale. %Hence for each $j=1,2,\ldots,D$, 
	%\begin{equation}	\label{eqn:man6compregimeswitchingjump}
	%	\tilde{\Phi}_{j}(t) = \Phi_{j}(t) - \zeta_{j}(t)
	%	\end{equation}
%	is an $(\{\mathcal{F}_{t}\}_{t\in [0,T]}, \mathbb{P})$-martingale.
%In the sequel, $\mathcal{M}^{2}(\mathbb{R}_{+}; \mathbb{R}^{D})$ is the set of functions $f(\cdot):\mathbb{R}_{+}\to\mathbb{R}^{D}$ such that $\| f(t) \|^{2}_{\mathcal{M}^{2}} := \sum\limits_{j=1}^{D} |f_{j}(t)|^{2}\zeta_{j}(t) < \infty$. 
Next, we define %some spaces of processes %that we will encounter going forward:
\begin{align*}
	&\mathcal{M}^{2}(\mathbb{R}_{+}; \mathbb{R}^{D}) := \{f(\cdot):\mathbb{R}_{+}\to\mathbb{R}^{D} \text{ s.t } \| f(t) \|^{2}_{\mathcal{M}^{2}} := \sum\limits_{j=1}^{D} |f_{j}(t)|^{2}\zeta_{j}(t) < \infty\} \\
	&L^{2}(\mathcal{F}_{T};\mathbb{R}) := \{ X\in\mathbb{R},\ \mathcal{F}_{T}\text{-measurable random variables, s.t.}\ \mathbb{E}\left[|X|^{2}\right] < \infty\};\\
	&S^{2}([0,T];\mathbb{R}) := \{ f\in\mathbb{R},\ \mathcal{F}_{t}\text{-adapted c\`adl\`ag, s.t.}\ \mathbb{E}[ \sup\limits_{0\le t\le T}|f(t)|^{2} ] < \infty \};\\
	&L^{2}_{\mathcal{F},p}([0,T];\mathbb{R}) := \{ f\in\mathbb{R},\text{predictable processes, s.t.}\ \mathbb{E}[ \int_{0}^{T} | f(t)| ^{2}\mathrm{d}t ] < \infty \};\\
	%F_{p}^{2}([0,T];\mathbb{R}^{M}) &:= \left\lbrace f\in\mathbb{R}^{M}\ \mathcal{F}_{t}\text{-predictable processes such that}\ \mathbb{E}\left[ \int_{0}^{T}\|f(t,\cdot)\|^{2}_{\mathcal{L}^{2}}\mathrm{d}t \right] < \infty \right\rbrace;\\
	&M^{2}_{p}([0,T];\mathbb{R}^{D}) := \{ f\in\mathbb{R}^{D},\text{predictable processes, s.t. }\mathbb{E}[ \int_{0}^{T}\|f(t,\cdot)\|^{2}_{\mathcal{M}^{2}}\mathrm{d}t ] < \infty \}.
\end{align*}
Let $A_1\subset \mathbb{R}$ be nonempty and $A_2 \subset [0,\infty)$. Denote by $U_1$ (resp. $U_2$) the class of measurable, adapted processes $u(t) = u(t,\omega) : [0,T]\times \Omega \rightarrow A_1$ (resp. $\xi(t) = \xi(t,\omega) : [0,T]\times \Omega \rightarrow A_2$ such that $\xi$ is nondecreasing, right-continuous with left limits and $\xi(0) = 0$).
%Let $U\subset \mathbb{R}^{K}$ be nonempty, and $u(t) = u(t,\omega):[0,T]\times\Omega\to U$ be a control process such that $\{u(t) | t\in[0,T]\}$ is $\left(\{\mathcal{F}_{t}\}_{t\in[0,T]}, \mathbb{P}\right)$-predictable with right limits.
We suppose the controlled state process $ X = \{X(t)\}_{t\in[0,T]}$ satisfies %has the following singular mean-field dynamics:
\begin{align}
	\label{eqn:man6meanfieldcontrolledstate}
	%\begin{split}
	\mathrm{d}X^{u,\xi}(t) =& b(t, X^{u,\xi}(t),\mathbb{E}[\varphi(X^{u,\xi}(t))], u(t), \alpha(t))\mathrm{d}t \notag\\
	&+ \sigma(t, X^{u,\xi}(t),\mathbb{E}[\varphi(X^{u,\xi}(t))], u(t), \alpha(t))\mathrm{d}B(t) + G(t, \alpha(t-))\mathrm{d}\xi(t)\notag\\
	%&+ \int_{\mathbb{R}_{0}^{M}}\eta(t, u(t-), \alpha(t-),z)\tilde{N}^{\alpha}(\mathrm{d}t, \mathrm{d}z) \\
	&+ \gamma(t, X^{u,\xi}(t-),\mathbb{E}[\varphi(X^{u,\xi}(t-))], u(t-), \alpha(t-))\mathrm{d}\tilde{\Phi}(t) \notag\\
	X^{u,\xi}(0) =& x_{0}\in\mathbb{R},
	%\end{split}
\end{align}
where $b:[0,T]\times\mathbb{R}\times\mathbb{R}\times A_1\times\mathcal{S} \to \mathbb{R}$, $\sigma:[0,T]\times\mathbb{R}\times\mathbb{R}\times A_1 \times\mathcal{S} \to \mathbb{R}$,
$\gamma:[0,T]\times\mathbb{R}\times\mathbb{R}\times A_1\times\mathcal{S} \to \mathbb{R}^{ D}$, and $G:[0,T]\times\mathcal{S}\to\mathbb{R}$ are given continuous functions, 
and $\tilde{\Phi} := \left(\tilde{\Phi}_{1}(t), \tilde{\Phi}_{2}(t), \ldots, \tilde{\Phi}_{D}(t)\right)$ with $\tilde{\Phi}_{j}(t)$ as given above.
%\eqref{eqn:man6compregimeswitchingjump}. 
We write $X(t) := X^{u,\xi}(t)$. %for the controlled state to emphasize the dependence on the controls. 
Furthermore, $\mathbb{E}[\cdot]$ denotes the expectation with respect to the probability measure $\mathbb{P}$, and $\varphi : \mathbb{R}\to \mathbb{R}$ is a Lipschitz continuous function. 
\begin{definition}
	An admissible control is an $\mathcal{F}_t$-adapted pair $(u,\xi)\in \mathcal{U}= U_1\times U_2$ such that \eqref{eqn:man6meanfieldcontrolledstate} has a unique strong solution and 
	\begin{equation}
		\mathbb{E}[\sup_{t\in [0,T]}|u(t)|^2 + |\xi(T)|^2] < \infty. %\,\,\,\, {\color{red}c_\infty=\|\int_0^T|\diffns \xi(t)|\|_{L^\infty}<\infty},
	\end{equation}
\end{definition}
%	From financial mathematics point of view, we can interpret the Brownian motion component of the controlled state process $\{X(t) | t\in[0,T]\}$ with dynamics given in \eqref{eqn:man6meanfieldcontrolledstate} to be the random fluctuations in a risky asset (stock index) which are caused by market events or speculative activities of investors having marginal impacts on the asset price. 
%Similarly, we can think of the controlled state $X^{u,\xi}(t)$ as having two kinds of jumps. The Hawkes random measure models jumps in the asset price which are triggered by the emergence of sudden market events or news having extraordinary impacts on the asset price. These impacts are known to be contagious in the sense that the jumps may occur in abrupt manner and in close succession \cite{ait2015portfolio}.
%These usually lead to excitation phenomenon, in that a jump in one asset class increases the chance of future jumps in the same asset class referred to as self-excitation or other asset classes termed mutual or cross excitation \cite{ait2015portfolio}. The Markov regime switching component is modeled by the basic martingales of the chain, and are attributed to transitions in economic conditions as a result of global as well as local economic challenges. The singular part models for example purchase or sales of stocks at time $t$.
Let us consider a performance criterion defined for each $x_{0}\in\mathbb{R}, e_{i}\in\mathcal{S}$ as
\begin{align}
	\label{eqn:man6performancecriterion}
	\begin{split}
		J(x_{0}, e_{i}; u,\xi) &:= \mathbb{E}_{x_{0},e_{i}}[ \int_{0}^{T}f(t,X(t),\mathbb{E}[\varphi(X(t))],u(t),\alpha(t)) \mathrm{d}t + \int_{0}^{T}\kappa(t)\mathrm{d}\xi(t)\\
		&\qquad\qquad+ h(X(T), \mathbb{E}[\varphi(X(T))],\alpha(T))],
	\end{split}
\end{align}
where $\mathbb{E}_{x_{0},e_{i}}$ denotes the conditional expectation given $X(0) = x_{0}$, and $\alpha(0) = e_{i}$ under the probability measure $\mathbb{P}$. Furthermore, the functions $f:[0,T]\times\mathbb{R}\times\mathbb{R}\times A_{1}\times\mathcal{S} \to \mathbb{R}$, $h:\mathbb{R}\times\mathbb{R}\times\mathcal{S}\to\mathbb{R}$, and $\kappa:[0,T]\to \mathbb{R}$ are given continuous functions. %\textcolor{red}{(What are the condition on $\varphi$?)}
\begin{proposition}
	We wish to find $(u^{*}, \xi^{*})\in \mathcal{U}$ such that 
	\begin{equation}
		\label{eqn:man6optimalcontroproblem}
		J(x_{0},e_{i};u^{*},\xi^{*}) = \sup\limits_{(u,\xi)\in \mathcal{U}}J(x_{0},e_{i};u, \xi).
	\end{equation}
	%	$(u^{*},\xi^{*})$ that solving \eqref{eqn:man6optimalcontroproblem} is called an \textsl{optimal control} and the corresponding controlled state $X^{u^{*},\xi^{*}}(t)$ is called an \textsl{optimal state}.
\end{proposition}
%In the sequel, we make give some standing assumptions that will be used. %assumptions on the coefficients. We denote the controlled state variable as $x$ and let $y = \mathbb{E}[\varphi(x)]$.
\begin{assumption}
	\label{assumpt:man6conditionsforexistence}
	The following assumptions will be used throughout this work.
	\begin{itemize}
		\item[$(\mathcal{C}1)$] The functions $b,\sigma, \gamma$ are uniformly Lipschitz continuous with respect to $(x,y)$, and of linear growth in $(x,y,u)$, i.e., there exists a constant $K>0$ such that 
		\begin{align*}
			\bullet &\ | b(t, x_{1}, y_{1}, u, e_{i})  -  b(t, x_{2}, y_{2}, u, e_{i}) | + | \sigma(t, x_{1}, y_{1}, u, e_{i}) -  \sigma(t, x_{2}, y_{2}, u, e_{i}) |\\
			&+ \| \gamma(t, x_{1}, y_{1}, u, e_{i})  -  \gamma(t, x_{2}, y_{2}, u, e_{i}) \|_{\mathcal{M}^{2}} \le K\left( | x_{1} - x_{2} | + | y_{1} - y_{2} | \right);\\
			\bullet &\ | b(t, x, y, u, e_{i})| + | \sigma(t, x, y, u, e_{i}) | + \| \gamma(t, x, y, u, e_{i}) \|_{\mathcal{M}^{2}}\\
			&\le K\left( 1 + | x | + | y | + | u | \right).
		\end{align*}
		
		\item[$(\mathcal{C}2)$] The functions $G$ and $\kappa$ (resp. $\varphi$) are continuous and bounded (resp. continuously differentiable with bounded derivative).
		
		\item[$(\mathcal{C}3)$] The functions $b, \sigma, \gamma, f, h$ are twice continuously differentiable with respect to $(x,y)$.
		
		\item[$(\mathcal{C}4)$] The derivatives of the functions $b, \sigma, \gamma$ are bounded, and the derivatives of $f$ are bounded by $K\left( 1 + | x | + | y | + | u |\right)$ and those of $h$ are bounded by $K\left( 1 + | x | + | y |  \right)$ for some constant $K>0$.
	\end{itemize}
\end{assumption}

\begin{lemma}%[Existence and uniqueness]
	\label{lem:man6existenceanduniqueness}
	Under \ref{assumpt:man6conditionsforexistence}, the %controlled Markov regime switching 
	SDE \eqref{eqn:man6meanfieldcontrolledstate} has a unique strong solution.
\end{lemma}

\begin{proof}
	The proof follows by combining arguments from \cite[Theorem 6, Chapter V]{protter2005stochastic} and \cite[Theorem 3.1]{shen2013maximum}.
\end{proof}
%Subsequently, the following notations will be used %throughout the remaining part of the work. 
For a vector-valued differentiable function $ \Gamma(t), $ we denote by $\Gamma_{x}(t)$ its derivative with respect to the spacial variable $x$. We write
\begin{align*}
	\delta b(t,\alpha(t)) :=&  b(t,X^{u^{*},\xi^{*}}(t),\mathbb{E}[\varphi(X^{u^{*},\xi^{*}}(t))],u,\alpha(t)) \\
	&- b(t,X^{u^{*},\xi^{*}}(t),\mathbb{E}[\varphi(X^{u^{*},\xi^{*}}(t))],u^{*}(t),\alpha(t)),\\
	b_x^\mu(t,\alpha(t)) :=&  b_x(t,X^{u^{*},\xi^{*}}(t)+\mu (X^{u^{\theta},\xi^{*}}(t)-X^{u^{*},\xi^{*}}(t)),\\
	&\mathbb{E}[\varphi(X^{u^{*},\xi^{*}}(t))]+\mu(\mathbb{E}[\varphi(X^{u^{\theta},\xi^{*}}(t))]-\mathbb{E}[\varphi(X^{u^{*},\xi^{*}}(t))]),u^\theta,\alpha(t)) ,
\end{align*}
and similarly for $\delta \sigma, \delta \gamma, \sigma^\mu_x$ and $\gamma^\mu_x$.

\section{Main results and proofs}
\label{sec:mainresultsandproofs}
In this section we state and prove the necessary and sufficient conditions of optimality. We define the Hamiltonian $H:[0,T]\times\mathbb{R}\times\mathbb{R}\times A_{1}\times\mathcal{S}\times\mathbb{R}\times\mathbb{R}\times\mathbb{R}^{D}\to\mathbb{R}$ as follows:
	\begin{align}
		\label{eqn:man6hamiltonianmeanfieldsingular}
		\begin{split}
			H(t, x, y, u, e_{i}, p, q, s)
			&=f(t, x, y, u, e_{i}) + b(t, x, y, u, e_{i})p + \sigma(t, x, y, u, e_{i})q \\
			&+ \sum\limits_{j=1}^{D}\gamma^{j}(t,x,y,u,e_{i})s_{j}(t)\zeta_{ij}(t).
		\end{split}
	\end{align}
The corresponding first-order adjoint processes $(p(t), q(t), s(t))\in L^{2}(\mathcal{F};\mathbb{R})\times L^{2}_{\mathcal{F}, p}([0,T]; \mathbb{R}) \times \mathcal{M}^{2}(\mathbb{R}_{+};\mathbb{R}^{D})$ associated to the Hamiltonian satisfies the following  BSDE: %orm solution to the following mean-field type singular backward stochastic differential equation (SBSDE):

	\begin{align}
		\label{eqn:man6adjointforxmeanfieldsingular}
		%\begin{split}
		\mathrm{d}p(t) =& -\big\{H_{x}(t,e_{i}) + \mathbb{E}\left[ H_{y}(t,e_{i}) \right]\varphi_{x}(X(t))\big\}\mathrm{d}t  + q(t)\mathrm{d}B(t) + s(t)\mathrm{d}\tilde{\Phi}(t)\notag\\
		=&-\big\{ f_{x}(t,e_{i}) + b_{x}(t,e_{i})p(t) + \sigma_{x} (t,e_{i})q(t)+ \sum\limits_{j=1}^{D}\gamma_{x}^{j}(t,e_{i})s_j(t)\zeta_{ij}(t) \notag\\
		&+\big( \mathbb{E}[f_{y}(t,e_{i})]+ \mathbb{E}[b_{y}(t,e_{i})p(t)] +\mathbb{E}[\sigma_{y}(t,e_{i})q(t)]\notag \\
		&+ \mathbb{E}[\gamma_{y}(t,e_{i})s(t)\zeta(t)]\big)\varphi_{x}(X(t))\big\}\mathrm{d}t  + q(t)\mathrm{d}B(t)  + s(t)\mathrm{d}\tilde{\Phi}(t),\\
		p(T) =& h_{x}(X(T), \mathbb{E}[\varphi(X(T))], \alpha(T)) + \mathbb{E}[h_{y}(X(T), \mathbb{E}[\varphi(X(T))], \alpha(T)) ]\varphi_{x}(X(T)).\notag
		%\end{split} 
	\end{align}
 
We also define the  corresponding second-order adjoint processes $(P(t), Q(t), S(t))\in L^{2}(\mathcal{F};\mathbb{R})\times L^{2}_{\mathcal{F}, p}([0,T]; \mathbb{R}) \times \mathcal{M}^{2}(\mathbb{R}_{+};\mathbb{R}^{D})$: %which is a solution to the following BSDE

	\begin{align}
		\label{eqn:man6secondorderadjointequationforP1}
		%\begin{split}
		\mathrm{d}P(t) =& -\Big(2b_{x}(t,e_{i})P(t) + ( \sigma_{x}(t,u,e_{i}))^{2}P(t) + 2\sigma_{x}(t,e_{i})Q(t)\notag\\
		&+ \sum\limits_{j=1}^D( \gamma^{j}_{x}(t)^{2}(P(t) + S_{j}(t)) + 2\gamma^{j}_{x}(t)S_j(t))\zeta_{ij}(t) + H_{xx}(t,e_{i})\Big)\mathrm{d}t \notag\\
		&+ Q(t)\mathrm{d}B(t) + S(t)\mathrm{d}\tilde{\Phi}(t),\\
		P(T) =& h_{xx}(X(T),\mathbb{E}[\varphi(X(T))], \alpha(T)).
	\end{align}

\subsection{Main results}
We now state the first main result of this section. 
\begin{theorem}[Necessary maximum principle]
	\label{thm:man6nmp}
	Suppose Assumptions \ref{assumpt:man6conditionsforexistence} %$(\mathcal{C}1)$--$(\mathcal{C}3^{\prime\prime})$ 
	hold. \\Let $(X^{u^{*},\xi^{*}}(\cdot), u^{*}(\cdot), \xi^{*}(\cdot))$ be an optimal solution to the control problem  \eqref{eqn:man6meanfieldcontrolledstate}--\eqref{eqn:man6optimalcontroproblem}. Then the first-order and second-order adjoint equations, \eqref{eqn:man6adjointforxmeanfieldsingular}--\eqref{eqn:man6secondorderadjointequationforP1}, admit unique solutions $(p(\cdot), q(\cdot),s(\cdot))$, and  $(P(\cdot), Q(\cdot), S(\cdot))$, respectively, and 
		\begin{align}
			\label{eqn:man6variationalinequality}
			%\begin{split}
			H&(t, X^{u^{*},\xi^{*}}(t), \mathbb{E}[\varphi(X^{u^{*},\xi^{*}}(t))],u,\alpha(t), p(t), q(t),s(t)) \notag\\
			&-  H(t, X^{u^{*},\xi^{*}}(t), \mathbb{E}[\varphi(X^{u^{*},\xi^{*}}(t))],u^{*}(t),\alpha(t), p(t), q(t),s(t))\notag\\
			&+ \frac{1}{2}P(t)\Big( \delta\sigma(t,\alpha(t))\Big)^{2} + \frac{1}{2}\sum\limits_{j=1}^{D}\Big(P(t) + S_j(t)\Big)\Big(\delta\gamma^{j}(t,\alpha(t),z) \Big)^{2}\zeta_{ij}(t)\\
			& \leq 0\ \forall u \in A_{1},\ \text{a.e}\ t\in [0,T], \mathbb{P}\text{-a.s}\notag
			%\end{split}
		\end{align}
	with 
	\begin{equation}
		\label{eqn:prob1}
		\mathbb{P}\{\forall t\in [0,T], (\kappa(t) + G(t, 
		\alpha(t-))p(t)) \le 0 \} = 1,
	\end{equation}
	\begin{equation}
		\label{eqn:prob2}
		\mathbb{P}( \mathbbm{1}_{\{\kappa(t) + G(t, \alpha(t-))p(t) \le 0\}}\mathrm{d}\xi^{*}(t) = 0 ) = 1.
	\end{equation}
	
	%$\mathbb{P}\{\forall t\in [0,T], (\kappa(t) + G(t, 
		%\alpha(t-))p(t)) \le 0 \} = 1$ 
	%and\\ $\mathbb{P}( \mathbbm{1}_{\{\kappa(t) + G(t, \alpha(t-))p(t) \le 0\}}\mathrm{d}\xi^{*}(t) = 0 ) = 1$.
	%	\[
	%	\mathbb{P}\Big\{ \forall t\in [0,T],\ \Big(\kappa(t) +  G(t,X^{u^{*},\xi^{*}}(t))p(t)\Big) \le 0 \Big\} = 1 \text{ and } \mathbb{P}\left( \mathbbm{1}_{\{\kappa(t) + G(t,X^{u^{*},\xi^{*}}(t))p(t) \le 0\}}\mathrm{d}\xi^{*}(t) = 0 \right) = 1.
	%	\]
\end{theorem}
\begin{remark}
	We prove Theorem \ref{thm:man6nmp} in two steps. In step 1 we derive the variational equations and associated moment estimates (Lemmas \ref{lem:man6lemma1}--\ref{lem:man6lemma5} and Proposition \ref{prop:man6propositionpage201}). The second step is devoted to deriving the duality relations between the adjoint processes and the variational equations (see Lemma \ref{man6lemma**} and  Proposition \ref{man6proposition4.6zhangsun}). %The proof will end with the derivation of the variational inequality \eqref{eqn:man6variationalinequality}.
\end{remark}
The next additional assumptions are needed for the sufficient conditions of optimality:
\begin{assumption}
	\label{assumpt:man6conditionsforsufficient}\leavevmode
	\begin{itemize}
		%The following assumptions will be used in formulating sufficient condition of optimality:
		\item[$(\mathcal{D}1)$] The control domain $A_1$ is convex.
		
		\item[$(\mathcal{D}2)$] The coefficients $b, \sigma, \gamma,f$ are differentiable with respect to $u$.
		
		%	\item[$(\mathcal{D}3)$] The Hamiltonian $H$ is differentiable with respect to $x, y, u$.
	\end{itemize}
\end{assumption}
%Next, we state the sufficient stochastic maximum principle for regime switching singular control.

\begin{theorem}[Sufficient maximum principle]
	\label{thm:man6sufficientconditionforsingularregimeswitchingmeanfield}
	Suppose Assumptions \ref{assumpt:man6conditionsforexistence} and \ref{assumpt:man6conditionsforsufficient} hold. Let $(u^{*}, \xi^{*}) \in \mathcal{U}$ be an arbitrary control and $X^{*}(t) := X^{u^{*},\xi^{*}}(t)$ be the corresponding controlled state process. Suppose there exist adapted solutions $(\hat{p}(t), \hat{q}(t), \hat{s}(t))$ to the corresponding first-order adjoint equation \eqref{eqn:man6adjointforxmeanfieldsingular} such that the following integrability conditions hold:
	
	\begin{align*}
		%\begin{split}
		&\mathbb{E}[ \int_{0}^{T}\left( X^{*}(t) - X(t) \right)\{ \hat{q}(t)\hat{q}(t)^{\top} + \hat{s}(t)\text{Diag}(\zeta(t))\hat{s}(t)^{\top}\}\left(X^{*}(t) - X(t)\right)\mathrm{d}t] < \infty,\\
		%\end{split}\\
		%		\begin{split}
			%			&\mathbb{E}\Bigg[ \int_{0}^{T}\left( \hat{\lambda}(t) - \lambda(t) \right)^{\top}\Big\{ \int_{\mathbb{R}_{0}}\hat{r}_{2}(t,z)\text{Diag}(\lambda(t)\nu^{e_{i}}(\mathrm{d}z)) \hat{r}_{2}(t,z)^{\top} \Big\}\times\left(\hat{\lambda}(t) - \lambda(t)\right)\mathrm{d}t \Bigg] < \infty,
			%		\end{split}\\
		%\begin{split}
		&\mathbb{E}[ \int_{0}^{T}p(t)\{ (\sigma\sigma)^{\top}(t,X(t), \mathbb{E}[\varphi(X(t))], u(t), \alpha(t))\} \hat{p}(t) \mathrm{d}t] < \infty.
		%\end{split}
	\end{align*}
	Furthermore, assume that the following conditions hold:
	\begin{enumerate}
		\item For each pair $(t, e_{i})\in [0,T]\times \mathcal{S}$, $H$, and  $h$ are concave functions in $x, y, u$.
		
		\item For almost all $t\in [0,T]$
		\begin{align}
			\label{eqn:man6maximumcondition}
			%\begin{split}
			&H(t, X^{*}(t), \mathbb{E}[\varphi(X^{*}(t))], u^{*}(t), e_{i}(t), p(t), q(t),s(t))\notag\\
			&= \sup\limits_{u\in A_{1}}H(t, X^{*}(t), \mathbb{E}[\varphi(X^{*}(t))], u, e_{i}(t), p(t), q(t), s(t)),\ \mathbb{P}\text{-a.s}.
			%\end{split}
		\end{align}
		
		\item 
		%			\textcolor{red}{
			%			\begin{align}
				%				\label{eqn:man6***}
				%				\int_{0}^{T}(\kappa(t) +  G(t, \alpha(t-))\mathrm{d}(\iota(t) - \xi^{*}(t)) \leq 0\ \forall u \in A_{1},\ \text{a.e}\ t\in [0,T],\ \mathbb{P}\text{-a.s}
				%			\end{align}
			%		}
		with 
			\begin{align}
				\label{probcond}
				\begin{split}
					&\mathbb{P}\{ \forall t \in [0,T],\left( \kappa(t) + G(t, \alpha(t-))p(t) \right) \le 0 \} = 1\\
					&\text{and}\\
					&\mathbb{P}( \mathbbm{1}_{\{\kappa(t) + G(t, \alpha(t-))p(t)\le 0\}}\mathrm{d}\xi^{*}(t)) = 1.
				\end{split}
			\end{align}
		%		\begin{align*}
			%			%\begin{split}
			%			\mathbb{P}\Big\{ \forall t \in [0,T],\ \left( \kappa(t) + G(t,X^{*}(t))p(t) \right) \le 0 \Big\} = 1,\ \text{and}\ \mathbb{P}\left( \mathbbm{1}_{\{\kappa(t) + G(t,X^{*}(t))p(t)\le 0\}}\mathrm{d}\xi^{*}(t) \right) = 1.
			%			%	\end{split}
		%	\end{align*}
\end{enumerate}

Then $(u^{*}, \xi^{*})$ is an optimal control and $X^{*}(\cdot)$ is the optimal state process.
\end{theorem}
We use spike variation for the regular control $u$, and convex perturbation for the singular control $\xi$ and define the perturbations as follows:

$$ (u^{\theta}(t), \xi^{\theta}(t)) := \begin{cases}
(u^{*}(t), \xi^{*}(t) + \theta(\iota(t) - \xi^{*}(t))), \ & t\in [0,T]\setminus E_{\theta},\\
(v, \xi^{*}(t) + \theta(\iota(t) - \xi^{*}(t))),\ & \text{otherwise},
\end{cases} $$
where $E_{\theta}\subset [0,T]$ is a measurable set with Lebesgue measure $|E_{\theta}| = \theta$, where $\theta > 0$ is sufficiently small, $v$ is an $A_{1}$-valued, $\mathcal{F}_{t}$-measurable random variable, and $\iota$ is a non-decreasing, left continuous with right limits and $\iota(0) = 0$. $u^{\theta}$ is the spike variation of the control process $u^{*}(\cdot)$. Since $(u^{*}, \xi^{*})$ is an optimal control pair, we have

\[
J(u^{\theta},\xi^{\theta}) - J(u^{*}, \xi^{*}) =  	J(u^{\theta},\xi^{\theta}) - J(u^{\theta},\xi^{*}) + J(u^{\theta},\xi^{*}) - J(u^{*}, \xi^{*})=J_{1}+J_{2} \le 0.
\]
Then the variational inequality will be obtained, if
\begin{equation}
\label{eqn:man6inequalityoflimits}
\lim\limits_{\theta\rightarrow 0}\frac{1}{\theta}J_{1} + \lim\limits_{\theta\rightarrow 0}\frac{1}{\theta}J_{2} \le 0.
\end{equation}
Let $X^{\theta,\xi^{\theta}}(t) := X^{u^{\theta}, \xi^{\theta}}(t),\ X^{\theta,\xi^{*}}(t) := X^{u^{\theta},\xi^{*}}(t)$ be the trajectories associated with the control pairs $(u^{\theta},\xi^{\theta})$, and $(u^{\theta}, \xi^{*})$, respectively. We compute the limits in \eqref{eqn:man6inequalityoflimits} separately. 
Next, we state some auxiliary results whose proofs are given later.
\begin{lemma}
\label{lem:man6lemma1}
Under Assumption \ref{assumpt:man6conditionsforexistence}, we have 

\begin{equation*}
	\lim_{\theta\rightarrow 0}\mathbb{E}[\sup_{t\in[0,T]}|\frac{X^{\theta,\xi^\theta}(t)-X^{\theta,\xi^{*}}(t)}{\theta}-Z(t)|^2] = 0,
\end{equation*}
where $Z$ is the solution to the following regime switching linear SDE
	\begin{align*}
		%\begin{split}
		\mathrm{d}Z(t) =& \{b_x(t)Z(t) + b_y(t)\mathbb{E}[\varphi_x(X(t))Z(t)]\}\mathrm{d}t \\ 
		&+ \{\sigma_x(t)Z(t)+\sigma_y(t)\mathbb{E}[\varphi_x(X(t))Z(t)]\}\mathrm{d}B(t) + G(t, \alpha(t-))\mathrm{d}(\iota-\xi^{*})(t)  \\
		&+ \{\gamma_x(t)Z(t) + \gamma_y(t)\mathbb{E}[\varphi_x(X(t))Z(t)]\}\mathrm{d}\tilde{\Phi}(t)\\
		Z(0) =& 0.
		%\end{split}
	\end{align*}
In addition, we have
	\begin{align}\label{pror1}
	\mathbb{E}[p(T)Z(T)] =&- \mathbb{E}[ \int_{0}^{T}Z(t)( f_{x}(t) + \mathbb{E}[f_{y}(t)]\varphi_{x}(X(t)) )\mathrm{d}t]\notag\\
	&+\mathbb{E}[\int_0^Tp(t)G(t,\alpha(t))\mathrm{d}(\iota-\xi^{*})(t),
\end{align}
\end{lemma}
\begin{proof}
	See Appendix \ref{sec:proofofauxi}.
\end{proof}
\begin{lemma}
\label{lem:man6lemma2}
Under \ref{assumpt:man6conditionsforexistence}, it holds
\begin{equation*}
	\mathbb{E}[\sup_{t\in[0,T]}|X^{\theta,\xi^{*}}(t)-X^{*}(t)-X_1(t)-X_2(t)|^2] \le C\theta^2\phi(\theta),
\end{equation*}
where $X_1$ and $X_2$ are solutions to the following SDEs
	\begin{align}
		\label{eqn:man6firstodervariationalequation}
		%\begin{split}
		\mathrm{d}X_1(t) =& \{b_x(t)X_1(t) + b_y(t)\mathbb{E}[\varphi_x(X(t))X_1(t)] + b(t,X^{*}(t),\mathbb{E}[\varphi(X^{*}(t))],u^\theta(t),\alpha(t))\notag\\
		&- b(t,X^{*}(t),\mathbb{E}[\varphi(X^{*}(t))],u^{*}(t),\alpha(t))\}\mathrm{d}t + \{\sigma_x(t)X_1(t)\notag\\
		&+ \sigma_y(t)\mathbb{E}[\varphi_x(X(t))X_1(t)] + \sigma(t,X^{*}(t),\mathbb{E}[\varphi(X^{*}(t))],u^\theta(t),\alpha(t)) \notag\\
		&- \sigma(t,X^{*}(t),\mathbb{E}[\varphi(X^{*}(t))],u^{*}(t),\alpha(t))\}\mathrm{d}B(t) + \{\gamma_x(t)X_1(t) \\
		%		&+ \int_{\mathbb{R}_{0}^{M}}\Big\{ \eta(t,u^{\theta}(t),z) - \eta(t,u^{*}(t),z) \Big\}\tilde{N}^{\alpha}(\mathrm{d}t, \mathrm{d}z)\\
		&+ \gamma_y(t)\mathbb{E}[\varphi_x(X(t))X_1(t)] + \gamma(t,X^{*}(t),\mathbb{E}[\varphi(X^{*}(t))],u^\theta(t),\alpha(t))\notag\\
		&- \gamma(t,X^{*}(t),\mathbb{E}[\varphi(X^{*}(t))],u^{*}(t),\alpha(t))\}\mathrm{d}\tilde{\Phi}(t),\notag\\
		X_{1}(0) =& 0,\notag
	\end{align}

	\begin{align}
		\label{eqn:man6secondordervariationalequation}
		%\begin{split}
		\mathrm{d}X_2(t) =& \{b_x(t,X^{*}(t),\mathbb{E}[\varphi(X^{*}(t))],u^\theta(t),\alpha(t))X_{1}(t)- b_x(t,X^{*}(t),\mathbb{E}[\varphi(X^{*}(t))],u^{*}(t),\notag\\
		&\alpha(t))X_1(t) + b_x(t,X^{*}(t),\mathbb{E}[\varphi(X^{*}(t))],u^{*}(t),\alpha(t))X_2(t)+ b_y(t,X^{*}(t),\notag\\
		&\mathbb{E}[\varphi(X^{*}(t))],u^{*}(t),\alpha(t))\mathbb{E}[\varphi_x(X(t))X_2(t)] + \frac{1}{2} b_{xx}(t,X^{*}(t),\mathbb{E}[\varphi(X^{*}(t))],\notag\\
		%				&+ \frac{1}{2} b_{\lambda\lambda}(t,X^{u^{*},\xi^{*}}(t),\mathbb{E}[\varphi(X^{u^{*},\xi^{*}}(t))],u^{*}(t),\alpha(t))X_1^2(t)\Big\}\mathrm{d}t\\
		&u^{*}(t),\alpha(t))X_1^2(t)\}\mathrm{d}t + \{\sigma_x(t,X^{*}(t),\mathbb{E}[\varphi(X^{*}(t))],u^\theta(t),\alpha(t))- \sigma_x(t, X^{*}(t),\notag\\
		&\mathbb{E}[\varphi(X^{*}(t))],u^{*}(t),\alpha(t))X_1(t) + \sigma_x(t,X^{*}(t),\mathbb{E}[\varphi(X^{*}(t))],u^{*}(t),\alpha(t))X_2(t) \notag\\
		&+ \sigma_y(t,X^{*}(t),\mathbb{E}[\varphi(X^{*}(t))],u^{*}(t),\alpha(t))\mathbb{E}[\varphi_x(X(t))X_2(t)]+ \frac{1}{2} \sigma_{xx}(t,X^{*}(t),\notag\\
		&\mathbb{E}[\varphi(X^{*}(t))],u^{*}(t),\alpha(t))\mathbb{E}[\varphi_x(X(t))X_1^2(t)]\}\mathrm{d}B(t)+ \{\gamma_x(t,X^{*}(t), \notag\\
		& \mathbb{E}[\varphi(X^{*}(t))],u^\theta(t),\alpha(t))- \gamma_x(t,X^{*}(t),\mathbb{E}[\varphi(X^{*}(t))],u^{*}(t),\alpha(t))X_1(t) \notag\\
		&+ \gamma_x(t,X^{*}(t),\mathbb{E}[\varphi(X^{*}(t))],u^{*}(t),\alpha(t))X_2(t)+ \gamma_y(t,X^{*}(t),\mathbb{E}[\varphi(X^{*}(t))],\notag\\
		&u^{*}(t),\alpha(t)) \times\mathbb{E}[\varphi_x(X(t))X_2(t)]+ \frac{1}{2} \gamma_{xx}(t,X^{*}(t),\mathbb{E}[\varphi(X^{*}(t))],u^{*}(t),\alpha(t))\notag\\
		&\times\mathbb{E}[\varphi_x(X(t))X_1^2(t)]\}\mathrm{d}\tilde{\Phi}(t),\, \, \, 
		X_{2}(0) = 0.
		%\end{split}
	\end{align}

\end{lemma}
%Under the assumptions of this lemma, the Lipschitz and linear growth conditions for the coefficients of \eqref{eqn:man6firstodervariationalequation} and \eqref{eqn:man6secondordervariationalequation} are satisfied. Hence, from Lemma \ref{lem:man6existenceanduniqueness}, $X_{1}$ and $X_{2}$ admit unique solutions.

%In what follows, we give some essential estimates that will be very useful in the proof of the necessary maximum principle.

\begin{lemma}
\label{lem:man6lemma3}
Under Assumption \ref{assumpt:man6conditionsforexistence}, let $\{\Psi(t)\ | \ t\in[0,T] \}$ be a process such that for any $p\geq 1$ 
\begin{equation*}
	\mathbb{E}[\sup_{t\in[0,T]}|\Psi(t)|^p] \le C.
\end{equation*}
	Then there exists $\phi: [0,\infty) \rightarrow [0,\infty)$ with $\phi(\theta)\downarrow 0$ as $\theta\downarrow 0$ such that 
	\begin{equation}
		\label{eqn:man6NO}
		\int_0^T(\mathbb{E}[\Psi(t)X_1(t)])^2\mathrm{d}t \le \theta \phi(\theta) 
		%\text{ and } \int_0^T(\mathbb{E}[\Psi(t)X_1(t)])^2\mathrm{d}\xi(t) \le \theta \phi_0(\theta).
	\end{equation} 

\end{lemma}
\begin{proof}
The proof proceeds as in \cite[Lemma 4.2]{zhang2018general}, using the differentiability of $\varphi$ and the boundedness of its derivative.
	\end{proof}

\begin{lemma}
\label{lem:man6lemma4}
Under Assumption \ref{assumpt:man6conditionsforexistence}, we have
\begin{align}
	\label{eqn:man6numberhash}
	\begin{split}
		&\sup_{t\in [0,T]}|\mathbb{E}[\varphi_x(X(t))X_1(t)]|^2 < \theta \hat{\rho}(\theta), 
		%\text{ where } \hat{\rho}: (0,\infty)\rightarrow (0,\infty) \text{ is such that } \hat{\rho} \downarrow 0 \text{ as } \theta \downarrow 0.
	\end{split}
\end{align}
where $\hat{\rho}: (0,\infty)\rightarrow (0,\infty)$ is such that $\hat{\rho} \downarrow 0$ as $\theta \downarrow 0$.
\end{lemma}

\begin{proof}
The proof follows as in \cite[Proposition 4.3]{zhang2018general}.
\end{proof}
\begin{lemma}
\label{lem:man6lemma5}
Under Assumption \ref{assumpt:man6conditionsforexistence}, it holds
\begin{equation}
	\mathbb{E}[\sup_{t\in [0,T]}|X^{(\theta,\xi^{*})}(t)-X^{*}(t)-X_1(t) - X_2(t)|^2] \le C \theta^2 \bar{\rho}(\theta), %\text{ where } \bar{\rho} : (0,\infty) \rightarrow (0,\infty) \text{ satisfies } \bar{\rho}(\theta)\downarrow 0 \text{  as } \theta\downarrow 0.
\end{equation}
where $\bar{\rho} : (0,\infty) \rightarrow (0,\infty)$ satisfies $\bar{\rho}(\theta)\downarrow 0$ as $\theta\downarrow 0$.
\end{lemma}
\begin{proof}
	See Appendix \ref{sec:proofofauxi}.
	\end{proof}
Let $F^{\theta}(T;h)$ be defined by
\begin{align*}
F^{\theta}(T;h) &:= h(X^{\theta,\xi^{*}}(T), \mathbb{E}[\varphi(X^{\theta,\xi^{*}}(T))],\alpha(T)) - h(X^{*}(T), \mathbb{E}[\varphi(X^{*}(T))],\alpha(T))\\
&+ h_{x}(T)(X_{1}(T) + X_{2}(T)) - h_{y}(T)\mathbb{E}[\varphi_{x}(\cdot)(X_{1}(T) + X_{2}(T))] \\
&- \frac{1}{2}h_{xx}(T)X_{1}^{2}(T).
\end{align*}
%\begin{proof}
%	The proof follows using a similar technique as in \cite{zhang2018general}.
%\end{proof}
Then we have the following estimate
% whose proof is similar to what has been demonstrated above.
\begin{proposition}	
\label{prop:man6propositionpage201}
\begin{align*}
	\mathbb{E}[| F^{\theta}(T;h)|] \leq \theta\bar{\bar{\rho}}(\theta) 	\text{ and } \bar{\bar{\rho}}(\theta) \to 0 \text{ as } \theta \downarrow 0
\end{align*}

\end{proposition} 
\begin{proof}
See Appendix \ref{sec:proofofauxi}.
\end{proof}
\begin{lemma}
\label{man6lemma**}
	\begin{align*}
		\mathbb{E}[p(T)X_{1}(T)] =& \mathbb{E}[ \int_{0}^{T}-X_{1}(t)( f_{x}(t) + \mathbb{E}[f_{y}(t)]\varphi_{x}(X^{\theta,\xi^{\theta}}(t)) )p(t)\mathrm{d}t]\\
		&+ \mathbb{E}[\int_{0}^{T}p(t)\delta b(t,u(t))\mathbbm{1}_{E_{\theta}}(t)\mathrm{d}t] + \mathbb{E}[ \int_{0}^{T}q(t)\delta \sigma(t,u(t))\mathbbm{1}_{E_{\theta}}(t)\mathrm{d}t]\notag\\
		&+ \mathbb{E}[ \int_{0}^{T}\sum\limits_{j=1}^{D}s_j(t)\delta \gamma^{j}(t,u(t))\zeta_{ij}(t)\mathbbm{1}_{E_{\theta}}(t)\mathrm{d}t],
	\end{align*}

and
	\begin{align*}
		\mathbb{E}[p(T)X_{2}(T)] =& \mathbb{E}[ \int_{0}^{T}-X_{2}(t)( f_{x}(t) + \mathbb{E}[f_{y}(t)]\varphi_{x}(X(t)) )\mathrm{d}t]+ \mathbb{E}[\int_{0}^{T} p(t)( \frac{1}{2}b_{xx}(t)\\
		%&+ \mathbb{E}\Big[\int_{0}^{T} p(t) \left( \frac{1}{2}b_{xx}(t)X_{1}^{2}(t) + \delta b_{x}(t,u(t))X_{1}(t)\mathbbm{1}_{E_{\theta}}(t) \right)\mathrm{d}t \Big]\\
		&\times X_{1}^{2}(t)+ \delta b_{x}(t,u(t))X_{1}(t)\mathbbm{1}_{E_{\theta}}(t))\mathrm{d}t]+ \mathbb{E}[\int_{0}^{T} q(t) ( \frac{1}{2}\sigma_{xx}(t)X_{1}^{2}(t)\\
		& + \delta \sigma_{x}(t,u(t))X_{1}(t)\mathbbm{1}_{E_{\theta}}(t) )\mathrm{d}t]+ \mathbb{E}[\int_{0}^{T}\sum\limits_{j=1}^{D}s_j(t)\Big( \frac{1}{2}\gamma^{j}_{xx}(t)X_{1}^{2}(t) \\
		&+ \delta \gamma^{j}_{x}(t,u(t))X_{1}(t)\mathbbm{1}_{E_{\theta}}(t)\Big)\zeta_{ij}(t)\mathrm{d}t ]
		%	\\
		%	\mathbb{E}[p_{2,j}(T)X_{2}(T)] =& \mathbb{E}\Big[ \int_{0}^{T}-X_{2}(t)\Big( \frac{\partial}{\partial\lambda_{j}}f(t)\Big)\mathrm{d}t \Big].
		%				&+ \mathbb{E}\Big[\int_{0}^{T} p_{2,j}(t) \left( \frac{1}{2}\frac{\partial^{2}}{\partial\lambda_{j}^{2}}b(t)X_{1}^{2}(t) + \delta \frac{\partial}{\partial\lambda_{j}}b(t,u(t))X_{1}(t)\mathbbm{1}_{E_{\theta}}(t) \right)\mathrm{d}t \Big].
	\end{align*}
\end{lemma}
\begin{proof}
The proof uses It\^o's formula and is similar to \cite[Theorem 4.1]{zhang2012stochastic}. %for brevity.
\end{proof}
We also have the following estimates needed for the proof of the variational inequality. %related to the solutions of the adjoint equations. %These estimates will be needed in the derivation of the variational inequality \eqref{eqn:man6variationalinequality}.
\begin{proposition}
\label{man6proposition4.6zhangsun} 
Suppose Assumption \ref{assumpt:man6conditionsforexistence}  holds. Then,
	\begin{align*}
		%\begin{split}
		&\mathbb{E}[ \int_{0}^{T}\{| p(t)\delta b_{x}(t,u(t))X_{1}(t)\mathbbm{1}_{E_{\theta}}(t)| + | q(t)\delta \sigma_{x}(t,u(t))X_{1}(t)\mathbbm{1}_{E_{\theta}}(t)|\\
		&\qquad+ \sum\limits_{j=1}^{D}|s_j(t)\delta \gamma^{j}_{x}(t,u(t))X_{1}(t)\mathbbm{1}_{E_{\theta}}(t)\zeta_{ij}(t)| \}\mathrm{d}t] \leq \theta\breve{\rho}(\theta),\\
		%\end{split}\\
		%			\begin{split}
			%				&\mathbb{E}\Big[ \int_{0}^{T}\Big\{ \Big| p_{2,j}(t)\delta \frac{\partial}{\partial\lambda_{j}}b(t,u(t))X_{1}(t)\mathbbm{1}_{E_{\theta}}(t)\Big\}\mathrm{d}t \Big] \leq \theta\breve{\rho}(\theta),
			%			\end{split}\\
		%\begin{split}
		&\mathbb{E}[\int_{0}^{T}\{|( P(t) b_{y}(t) + P(t)\sigma_{x}(t)\sigma_y(t) + Q(t)\sigma_y(t))X_{1}(t)\mathbb{E}[\varphi(X_{1}(t))]| \\
		&\qquad + \sum\limits_{j=1}^{D}|(( P(t) + S_j(t) )\gamma^{j}_{x}(t)\gamma^{j}_y(t) + S_j(t)\gamma_y^{j}(t) )X_{1}(t)\mathbb{E}[\varphi(X_{1}(t))]\zeta_{ij}(t)|\}\mathrm{d}t] \leq \theta\breve{\rho}(\theta),\\
		%\end{split}\\
		%\begin{split}
		&\mathbb{E}[ \int_{0}^{T}\{| P(t)\sigma_y^{2}(t) (\mathbb{E}[\varphi(X_{1}(t))])^{2}|  + \sum\limits_{j=1}^{D} |(P(t) + S_j(t))\gamma^{j}_y(t)^{2} (\mathbb{E}[\varphi(X_{1}(t))])^{2}\zeta_{ij}(t)|\}\mathrm{d}t] \\
		&\leq \theta\breve{\rho}(\theta),
		%\end{split}
	\end{align*}
where $\breve{\rho}(\cdot): (0,\infty) \to (0,\infty)$ satisfies $\breve{\rho}(\theta) \to 0$ as $\theta \downarrow 0$.
\end{proposition}

\begin{proof}
The proof is similar to \cite[Proposition 4.6]{zhang2018general}. It is therefore omitted here.
\end{proof}

%The proofs of \cref{lem:man6lemma1}, \cref{lem:man6lemma3}, \cref{lem:man6lemma5} and \cref{prop:man6propositionpage20} are provided in full version on {\color{red}arXiv}.

\subsection{Proofs of the main results}
Here, we provide the proofs for the main results. 

\begin{proof}(Proof of Theorem \ref{thm:man6nmp})
 Assume $(X^{*}(\cdot), u^{*}(\cdot), \xi^{*}(\cdot))$ is an optimal solution to   \eqref{eqn:man6meanfieldcontrolledstate}--\eqref{eqn:man6optimalcontroproblem}. We prove the theorem using \eqref{eqn:man6inequalityoflimits}.
	From the definition of the performance criterion \eqref{eqn:man6performancecriterion}, we have

\begin{align*}
	J_{1} =& J(u^{\theta}, \xi^{\theta}) - J(u^{\theta},\xi^{*}) = \mathbb{E}[ \int_{0}^{T}\{ \int_{0}^{1}f_{x}( t, X^{\theta,\xi^{*}}(t) + \vartheta (X^{\theta,\xi^{\theta}}(t) - X^{\theta,\xi^{*}}(t)) \\
	&,\mathbb{E}[\varphi( X^{u^{\theta},\xi^{*}}(t) + \vartheta(X^{\theta,\xi^{\theta}}(t) - X^{\theta,\xi^{*}}(t)))], u^{\theta}(t),\alpha(t)) \\
	&\times (X^{\theta, \xi^{\theta}}(t) - X^{\theta,\xi^{*}}(t))\mathrm{d}\vartheta +\int_{0}^{1}f_{y}( t, X^{\theta,\xi^{*}}(t) + \vartheta(X^{\theta,\xi^{\theta}}(t) - X^{\theta,\xi^{*}}(t)),\\
	&,\mathbb{E}[\varphi( X^{\theta,\xi^{*}}(t) + \vartheta(X^{\theta,\xi^{\theta}}(t) - X^{\theta,\xi^{*}}(t)))], u^{\theta}(t), \alpha(t)) \\
	&\times\mathbb{E}[\varphi_{x}( X^{\theta,\xi^{*}}(t) + \vartheta (X^{\theta,\xi^{\theta}}(t) - X^{\theta,\xi^{*}}(t)) \cdot(X^{\theta, \xi^{\theta}}(t) - X^{\theta,\xi^{*}}(t))]\mathrm{d}\vartheta\}\mathrm{d}t\\
	&+ \int_{0}^{1}h_{x}^{\theta}(T)(X^{\theta, \xi^{\theta}}(T) - X^{\theta,\xi^{*}}(T))\mathrm{d}\vartheta + \int_{0}^{T}\kappa(t)\mathrm{d}(\xi^{\theta}(t) - \xi^{*}(t)) \\ 
	&+ \int_{0}^{1}h_{y}^{\theta}(T)\mathbb{E}[\varphi_{x}( X^{\theta,\xi^{*}}(t) +  \vartheta(X^{\theta,\xi^{\theta}}(t) - X^{\theta,\xi^{*}}(t))) (X^{\theta, \xi^{\theta}}(t) - X^{\theta,\xi^{*}}(t)) ]\mathrm{d}\vartheta].
\end{align*}
Dividing by $\theta$ and using Lemma \ref{lem:man6lemma1} and  \eqref{eqn:man64.1} together with the fact that derivatives of $f$ are bounded, and that of $h$ are bounded by $C_{1}\left(1 + |x|^{p} + |y|^{p} + |u|^{p}\right)$, we have

\begin{align*}
	\lim\limits_{\theta\to 0}\frac{J_{1}}{\theta} 
	=& \mathbb{E}[\int_{0}^{T}f_{x}(t,X^{*}(t),\mathbb{E}[\varphi(X^{*}(t))],u^{*}(t),\alpha(t))Z(t) \mathrm{d}t] \\
	&+ \mathbb{E}[\int_{0}^{T}f_{y}(t, X^{*}(t), \mathbb{E}[\varphi(X^{*}(t))],u^{*}(t),\alpha(t))\\
	&\times \mathbb{E}[\varphi_{x}(X^{*}(t)Z(t))]\mathrm{d}t] + \mathbb{E}[p(T)Z(T)] + \mathbb{E}[\int_{0}^{T}\kappa(t)\mathrm{d}(\iota(t) - \xi^{*}(t))].
\end{align*}
Using \eqref{pror1}, we have
\begin{align}
	\label{eqn:man6page211}
	\lim\limits_{\theta \to 0}\frac{J_{1}(\theta)}{\theta} =& \mathbb{E}[ \int_{0}^{T}f_{x}(t,X^{*}(t),\mathbb{E}[\varphi(X^{*}(t))],u^{*}(t),\alpha(t))Z(t) \mathrm{d}t]\notag \\
	&+ \mathbb{E}[ \int_{0}^{T}f_{y}( t, X^{*}(t), \mathbb{E}[\varphi(X^{*}(t))],u^{*}(t),\alpha(t))\mathbb{E}[\varphi_{x}(X^{*}(t)Z(t))]\mathrm{d}t]\notag \\
	&- \mathbb{E}[ \int_{0}^{T}Z(t)\{ f_{x}(t) + \mathbb{E}[f_{y}(t)]\varphi_{x}(X(t))\}\mathrm{d}t] \\ 
	& + \mathbb{E}[ \int_{0}^{T}(\kappa(t) + p(t)G(t,\alpha(t-))\mathrm{d}(\iota(t)- \xi^{*}(t))] \notag \\
	=& \mathbb{E}[\int_{0}^{T} (\kappa(t) + p(t)G(t, \alpha(t-))\mathrm{d}(\iota(t)- \xi^{*}(t))].
\end{align}
We establish the second part following the approach in \cite{zhang2018general}. The main difference is that the coefficients in our setting depend on 
$\mathbb{E}[\varphi(X(t))]$ rather than on $\mathbb{E}[X(t)]$.
	\begin{align*}
		0 \geq&
		J_2= J(x_{0},e_{i}; u^{\theta}, \xi^{*}) - J(x_{0}, e_{i}; u^{*}, \xi^{*})\\
		=& \mathbb{E}[ p(T)(X_{1}(T) + X_{2}(T))]+ \mathbb{E}[ \int_{0}^{T}\{ \delta f(t,u(t)) + \frac{1}{2} f_{xx}(t)X_{1}^{2}(t)\}\mathrm{d}t] + o(\theta)  \\
		&+ \frac{1}{2}\mathbb{E}[ h_{xx}(T)X_{1}^{2}(T)] + \int_{0}^{T}\mathbb{E}[( f_{x}(t) + \mathbb{E}[ f_{y}(t) ]\varphi_{x}(X^{*}(t)) )(X_{1}(t) + X_{2}(t))]\mathrm{d}t.
	\end{align*}
	Using Lemma \ref{man6lemma**}, we have
	\begin{align*}
			&J(x_{0}, e_{i}; u^{\theta}, \xi^{*}) - J(x_{0}, e_{i}; u^{*}, \xi^{*})\\
			=& \mathbb{E}[ -\int_{0}^{T}(X_{1}(t) + X_{2}(t))( f_{x}(t) + \mathbb{E}[f_{y}(t)]\varphi_{x}(X(t)))\mathrm{d}t ]+ \mathbb{E}[ \int_{0}^{T}p(t)\{ ( \delta b(t,u(t))\\
			&+ \delta b_{x}(t,u(t))X_{1}(t)  )\mathbbm{1}_{E_{\theta}}(t) + \frac{1}{2}b_{xx}(t)X_{1}^{2}(t) \}\mathrm{d}t ]+ \mathbb{E}[ \int_{0}^{T}q(t)\{( \delta \sigma(t,u(t)) + \delta \sigma_{x}(t) \\
			&\times X_{1}(t) )\mathbbm{1}_{E_{\theta}}(t) + \frac{1}{2}\sigma_{xx}(t)X_{1}^{2}(t)\}\mathrm{d}t ]+ \mathbb{E}[ \int_{0}^{T}\sum\limits_{j=1}^{D}s_j(t)\Big\{( \delta \gamma^{j}(t,u(t)) + \delta \gamma^{j}_{x}(t)X_{1}(t) )\mathbbm{1}_{E_{\theta}}(t)\\
			%&+ \mathbb{E}\Big[ \int_{0}^{T}s(t)\Big\{ \Big( \delta \gamma(t,u(t)) + \delta \gamma_{x}(t)X_{1}(t) \Big)\mathbbm{1}_{E_{\theta}}(t) + \frac{1}{2}\gamma_{xx}(t)X_{1}^{2}(t) \Big\}\zeta(t)\mathrm{d}t \\
			& + \frac{1}{2}\gamma^{j}_{xx}(t)X_{1}^{2}(t)\Big\}\zeta_{ij}(t)\mathrm{d}t ]+ \frac{1}{2}\mathbb{E}[ h_{xx}(T)X_{1}^{2}(T)]\\
			& + \mathbb{E}[ \int_{0}^{T}( f_{x}(t) + \mathbb{E}[f_{y}(t)]\varphi_{x}(X^{*}(t)))(X_{1}(t) + X_{2}(t))\mathrm{d}t] + \mathbb{E}[\int_{0}^{T}\{ \delta f(t,u(t))\mathbbm{1}_{E_{\theta}}(t)\\
			&+ \frac{1}{2}f_{xx}(t)X_{1}^{2}(t) \}\mathrm{d}t] + o(\theta)\\
			%&+ \mathbb{E}\Big[ \int_{0}^{T}\Big( f_{x}(t) + \mathbb{E}[f_{y}(t)]\varphi_{x}(X^{*}(t)) \Big) \Big(X_{1}(t) + X_{2}(t)\Big)\mathrm{d}t\Big]\\
			%&+ \mathbb{E}\Big[\int_{0}^{T}\Big\{ \delta f(t,u(t))\mathbbm{1}_{E_{\theta}}(t) + \frac{1}{2}f_{xx}(t)X_{1}^{2}(t) \Big\}\mathrm{d}t\Big] + o(\theta)\\
			=& \mathbb{E}[ \int_{0}^{T} \Big(\delta H(t,u(t))\mathbbm{1}_{E_{\theta}}(t) + \frac{1}{2}H_{xx}(t)X_{1}^{2}(t)\Big)\mathrm{d}t + \frac{1}{2}\mathbb{E}[h_{xx}(T)X_{1}^{2}(T)] + o(\theta).
		\end{align*}
	It\^o's formula for semimartingales (see, for example,  \cite{protter2005stochastic}) gives
		\begin{align*}
			\mathrm{d}X_{1}^{2}(t) %=& 2X_{1}(t)\mathrm{d}X_{1}(t) + \mathrm{d}\left\langle X_{1} \right\rangle_{t}\\
			&= 2X_{1}(t)\Big\{\Big( b_{x}(t)X_{1}(t) %+ b_{\lambda}(t)X_{1}(t) 
			+ b_{y}(t)\mathbb{E}[\varphi_{x}(X(t))X_{1}(t)] + \delta b(t,u(t))\mathbbm{1}_{E_{\theta}}(t)\Big)\mathrm{d}t\\
			&+ \Big( \sigma_{x}(t)X_{1}(t) + \sigma_y(t)\mathbb{E}[\varphi_{x}(X(t))X_{1}(t)] + \delta\sigma(t,u(t))\mathbbm{1}_{E_{\theta}}(t)\Big)\mathrm{d}B(t)\\
			&+\Big(\gamma_x(t)X_{1}(t) + \gamma_y(t)\mathbb{E}[\varphi_{x}(X(t))X_{1}(t)] + \delta\gamma(t,u(t))\mathbbm{1}_{E_{\theta}}(t) \Big)\mathrm{d}\tilde{\Phi}(t)\Big\}\\
			&+ \Big( \sigma_{x}(t)X_{1}(t) + \sigma_y(t)\mathbb{E}[\varphi_{x}(X(t))X_{1}(t)] + \delta\sigma(t,u(t))\mathbbm{1}_{E_{\theta}}(t)\Big)^{2}\mathrm{d}t \\
			&+\sum\limits_{j=1}^{D} \Big(\gamma^{j}_{x}(t)X_{1}(t) + \gamma^{j}_y(t)\mathbb{E}[\varphi_{x}(X(t))X_{1}(t)] + \delta\gamma^{j}(t,u(t))\mathbbm{1}_{E_{\theta}}(t) \Big)^{2}\zeta_{ij}(t)\mathrm{d}t.
		\end{align*}
	Thus, we have
		\begin{align*}
			&\mathbb{E}[ h_{xx}(T)X_{1}^{2}(T)]\\
			&= \mathbb{E}[ P(0)X_{1}^{2}(0) + \int_{0}^{T}P(t)\mathrm{d}X_{1}^{2}(t) + \int_{0}^{T}X_{1}^{2}(t)\mathrm{d}P(t) + \int_{0}^{T}\mathrm{d}[ P, X_{1}^{2}](t)]\\%\mathbb{E}\Big[ P(T)X_{1}^{2}(T) \Big] \\
			&= \mathbb{E}[ \int_{0}^{T}2X_{1}(t)\{ P(t)b_{y}(t) + \sigma_{x}(t)\sigma_{y}(t) + \sigma_{y}(t)Q(t) \}\mathbb{E}[\varphi_{x}(X(t))X_{1}(t)]\mathrm{d}t] \\
			%&+ \mathbb{E}\Big[\int_{0}^{T}2X_{1}(t)\Big\{ P(t)\gamma_{x}(t)\gamma_{y}(t) + S(t)\gamma_{y}(t) \Big\}\zeta(t)\mathbb{E}[\varphi_{x}(X(t))X_{1}(t)]\mathrm{d}t\Big]\\
			& + \mathbb{E}[\int_{0}^{T}2X_{1}(t)\sum\limits_{j=1}^{D}\Big\{P(t)\gamma^{j}_{x}(t)\gamma^{j}_{y}(t) + S_j(t)\gamma^{j}_{y}(t) \Big\}\zeta_{ij}(t)\mathbb{E}[\varphi_{x}(X(t))X_{1}(t)]\mathrm{d}t]\\
			&+ \mathbb{E}[\int_{0}^{T}\sum\limits_{j=1}^{D}P(t)\{ \sigma_{y}^{2}(t) + \gamma^{j}_{y}(t)^{2}\zeta_{ij}(t)\}( \mathbb{E}[\varphi_{x}(X(t))X_{1}(t)])^{2}\mathrm{d}t] + \mathbb{E}[ \int_{0}^{T}\sum\limits_{j=1}^{D}P(t) \\
			&\times\{ (\delta\sigma(t,u(t)))^{2} + (\delta\gamma^{j}(t,u(t)))^{2}\zeta_{ij}(t) \}(\mathbbm{1}_{E_{\theta}}(t))^{2}\mathrm{d}t ]- \mathbb{E}[\int_{0}^{T}X_{1}^{2}(t)H_{xx}(t)\mathrm{d}t]\\
			&+ \mathbb{E}[2\int_{0}^{T}\sum\limits_{j=1}^{D}P(t)\Big\{ X_{1}(t) \delta b(t,u(t)) + X_{1}(t) \delta\sigma(t,u(t)) + X_{1}(t)(\gamma^{j}_{x}(t)\delta\sigma(t,u(t)) + \gamma^{j}_{y}(t)\\
			&\times \delta\gamma^{j}(t,u(t)))\zeta_{ij}(t) + (\sigma_{y}(t) + \gamma^{j}_{y}(t)\delta\gamma^{j}(t,u(t))\zeta_{ij}(t)) \mathbb{E}[\varphi_{x}(X(t))X_{1}(t)]\Big\}\mathbbm{1}_{E_{\theta}}(t)\mathrm{d}t]\\
			%&+ \mathbb{E}\Big[2\int_{0}^{T}X_{1}(t)\Big\{ Q(t) \delta\sigma(t,u(t)) + S(t)\delta\gamma(t,u(t))\zeta(t) \Big\}\mathbbm{1}_{E_{\theta}}(t)\mathrm{d}t\Big]\\
			&+ \mathbb{E}[2\int_{0}^{T}\sum\limits_{j=1}^{D}X_{1}(t)\Big\{ Q(t) \delta\sigma(t,u(t)) + S_j(t)\delta\gamma^{j}(t,u(t))\zeta_{ij}(t) \Big\}\mathbbm{1}_{E_{\theta}}(t)\mathrm{d}t] + \mathbb{E}[\int_{0}^{T}\sum\limits_{j=1}^{D}S_j(t)\\
			&\times\Big\{( \gamma^{j}_{x}(t)X_{1}(t) )^{2} + (\gamma^{j}_{y}(t))^{2}(\mathbb{E}[\varphi_{x}(X(t))X_{1}(t)])^{2}+ (\delta\gamma^{j}(t,u(t))\mathbbm{1}_{E_{\theta}}(t))^{2} \\
			&+ 2\gamma^{j}_{x}(t)X_{1}(t)\gamma^{j}_{y}(t)\mathbb{E}[\varphi_{x}(X(t))X_{1}(t)] + 2\gamma^{j}_{x}(t)X_{1}(t)\delta\gamma^{j}(t,u(t))\mathbbm{1}_{E_{\theta}}(t)\\
			&+ 2\gamma^{j}_{y}(t)\mathbb{E}[\varphi_{x}(X(t))X_{1}(t)]\delta\gamma^{j}(t,u(t))\mathbbm{1}_{E_{\theta}}(t)\Big\}\zeta_{ij}(t)\mathrm{d}t].
		\end{align*}
Using once more Proposition \ref{man6proposition4.6zhangsun}, the above expression is reduced to
		\begin{align*}
			\mathbb{E}[h_{xx}(T)X_{1}^{2}(T)] =& \mathbb{E}[\int_{0}^{T} P(t)(\delta\sigma(t,u(t)))^{2} \mathbbm{1}_{E_{\theta}}(t)\mathrm{d}t] - \mathbb{E}[\int_{0}^{T}X_{1}^{2}(t)H_{xx}(t)\mathrm{d}t] \\
			& + \mathbb{E}[\int_{0}^{T}\sum\limits_{j=1}^{D}\Big( P(t) + S_j(t)\Big)(\delta\gamma^{j}(t,u(t)))^{2}\zeta_{ij}(t) \mathbbm{1}_{E_{\theta}}(t)\mathrm{d}t]  + o(\theta).
		\end{align*}
Thus,
		\begin{align}
			%\begin{split}
			\label{eqn:man6**}
			&0 \geq J(x_{0},e_{i}; u^{\theta}, \xi^{*}) - J(x_{0}, e_{i}; u^{*}, \xi^{*})= \mathbb{E}[\int_{0}^{T}\{ \frac{1}{2}P(t)(\delta\sigma(t,u(t)))^{2} \\\notag
			&+ \sum\limits_{j=1}^{D}(P(t) + S_j(t))(\delta\gamma^{j}(t,u(t)))^{2}\zeta_{ij}(t) + \delta H(t,u(t)) \}\mathbbm{1}_{E_{\theta}}(t)\mathrm{d}t] + o(\theta).
			%\end{split}
		\end{align}
Dividing both sides of \eqref{eqn:man6**} by $\theta$ and letting $\theta \rightarrow 0$ and combining the results with \eqref{eqn:man6page211} yields
		\begin{align*}
			%\begin{split}
			0 \geq &\mathbb{E}[ H(t, X^{*}(t), \mathbb{E}[\varphi(X^{*}(t))], u(t),\alpha(t), p(t), q(t), s(t))\\
			&- H(t, X^{*}(t), \mathbb{E}[\varphi(X^{*}(t))],u^{*}(t),\alpha(t), p(t), q(t), s(t))+ \frac{1}{2}P(t)(\delta\sigma(t,u(t)))^{2}\\
			& + \frac{1}{2}\sum\limits_{j=1}^{D}\Big(P(t) + S_j(t)\Big)(\delta\gamma^{j}(t,u(t)))^{2}\zeta_{ij}(t)] \\
			&+ \mathbb{E}[ \int_{0}^{T}\{ \kappa(t) + G(t, \alpha(t-))p(t) \}\mathrm{d}(\iota(t) - \xi^{*}(t) )]. 
			%\end{split}
		\end{align*}
From the above, it hold $\mathbb{P}$-a.s that for a.e $t\in[0,T],\, \forall u\in U_{1}$
		\begin{align*}
			%\begin{split}
			0\geq& H(t, X^{*}(t), \mathbb{E}[\varphi(X^{*}(t))],u(t),\alpha(t), p(t),q(t), s(t))\\
			&- H(t, X^{*}(t), \mathbb{E}[\varphi(X^{*}(t))], u^{*}(t),\alpha(t), p(t), q(t),s(t))+\frac{1}{2}P(t)(\delta\sigma(t,u(t)))^{2}\\
			&+   \frac{1}{2}\sum\limits_{j=1}^{D}\Big(P(t) + S_j(t)\Big)(\delta\gamma^{j}(t,u(t)))^{2}\zeta_{ij}(t)\\
&			+ \int_{0}^{T}\{ \kappa(t) + G(t, \alpha(t-))p(t) \}\mathrm{d}(\iota(t) - \xi^{*}(t) ) 
			%\end{split}
		\end{align*}
Substituting $\iota(t) = \xi^{*}(t)$, we obtain \eqref{eqn:man6variationalinequality}. In addition, if we choose $u(t) = u^{*}(t)$ and proceed as in \cite[Theorem 4.2]{cadenillas1994stochastic}, we deduce \eqref{eqn:prob1} and \eqref{eqn:prob2}.		
\end{proof}
Now we proceed to the proof of the sufficient maximum principle.

\begin{proof}[Proof of Theorem \ref{thm:man6sufficientconditionforsingularregimeswitchingmeanfield}]
	For any $(u(\cdot), \xi(\cdot))\in \mathcal{U}$, consider the difference
	\begin{align*}
		&J(x_{0}, e_{i}, u(\cdot), \xi(\cdot)) - 	J(x_{0}, e_{i}, u^{*}(\cdot), \xi^{*}(\cdot))\\
		=& \mathbb{E}[ \int_{0}^{T}\{ f(t, e_i) - f^{*}(t,e_i)\}\mathrm{d}t]+ \mathbb{E}[\int_{0}^{T}\kappa(t)\mathrm{d}(\iota(t) - \xi^{*}(t))]\\
		&+ \mathbb{E}[ h(X(T), \mathbb{E}[\varphi(X(T))], \alpha(T)) - h(X^{*}(T), \mathbb{E}[\varphi(X^{*}(T))], \alpha(T))] \\
		=&  I_{1} + I_{2} + I_{3}.
	\end{align*}
where $f^{*}(t,e_i)=f(t,X^{*}(t),\mathbb{E}[\varphi(X^{*}(t))],u^*(t), e_i)$ and $f(t,e_i)$ is define similarly with $X$ in lieu of $X^*$. 	Using the concavity of $h$, we have
	
	\begin{align*}
		I_{2} \leq& \mathbb{E}[ h_{x}(T)(X(T) - X^{*}(T)) + h_{y}(T)(\mathbb{E}[\varphi_{x}(X(T))(X(T) - X^{*}(T))])] \\
		=& \mathbb{E}[p(T)(X(T) - X^{*}(T))].
	\end{align*}
	Now, applying It\^{o}'s product rule to $p(T)(X(t) - X^{*}(T))$, we get
	\begin{align*}
		%\begin{split}
		&\mathbb{E}[p(T)(X(T) - X^{*}(T))] \\
		=& \mathbb{E}\big[-\int_{0}^{T}(X(t) - X^{*}(t))\big\{ f_{x}^*(t,e_{i}) + b_{x}^*(t,e_i)p(t)+ \sigma_{x}^*(t,e_{i})q(t)  \\ 
		&+ \gamma_{x}^*(t,e_i)s(t)\zeta(t) +\big( \mathbb{E}[f^*_{y}(t,e_{i})] + \mathbb{E}[b^*_{y}(t,e_i)p(t)]+ \mathbb{E}[\gamma_{y}^*(t,e_i)s(t)\zeta(t)]\big)\varphi_{x}(X(t))\big\}\mathrm{d}t \big] \\
		&+ \mathbb{E}\big[\int_{0}^{T} p(t)\{(b(t,e_i) - b^*(t,e_i))\mathrm{d}t + G(t, e_i)\mathrm{d}\xi(t) - G(t,e_i)\diffns \xi^{*}(t)\}\big]\\
		&+ \mathbb{E}\big[ \int_{0}^{T}\big(\sigma(t,e_i)  - \sigma^*(t,e^i)\big)q(t)\mathrm{d}t\big] + \mathbb{E}[\int_{0}^{T}\sum\limits_{j=1}^{D}\big\{(\gamma^{j}(t,e_i)
		- \gamma^{j,*}(t,e_i))s_j(t)\zeta_{ij}(t)\big\}\mathrm{d}t].
		%\end{split}
	\end{align*}
Using the definition of the Hamiltonian, we have
	\begin{align*}
		I_{1} =& \mathbb{E}[ \int_{0}^{T}\{ f(t, X(t), \mathbb{E}[\varphi(X(t))], u(t),\alpha(t)) - f(t, X^{*}(t), \mathbb{E}[\varphi(X^{*}(t))], u^{*}(t),\alpha(t))\}\mathrm{d}t]\\
		=& \mathbb{E}[\int_{0}^{T}\sum\limits_{j=1}^{D}\{ H(t,\alpha(t)) - H^{*}(t,\alpha(t)) + (b^{*}(t,\alpha(t)) - b(t,\alpha(t)))p(t) \\ 
		&+ (\sigma^{*}(t,\alpha(t)) - \sigma(t,\alpha(t)))q(t) + (\gamma^{j*}(t,\alpha(t)) - \gamma^{j}(t,\alpha(t)))s_j(t)\zeta_{ij}(t)\}\mathrm{d}t ].
	\end{align*}
In $I_{1}$ above, we have used the following shorthand notations:

\begin{align*}
	H(t,\alpha(t)) &= H(t, X(t), \mathbb{E}[\varphi(X(t))], u(t), \alpha(t), p(t), q(t), s),\\
	H^{*}(t,\alpha(t)) &= H(t, X^{*}(t), \mathbb{E}[\varphi(X^{*}(t))], u^{*}(t), \alpha(t), p(t), q(t), s),\\
%	b(t,\alpha(t)) &= b(t, X(t), \mathbb{E}[\varphi(X(t))], u(t), \alpha(t)),\\
%	b^{*}(t,\alpha(t)) &= b(t, X^{*}(t), \mathbb{E}[\varphi(X^{*}(t))], u^{*}(t), \alpha(t)),
\end{align*}
and similarly for $b,b^*,\sigma, \sigma^{*}, \gamma$ and $\gamma^{*}$. Putting everything together, we get

	\begin{align*}
		&I_{1} + I_{2} + I_{3}\\
		\leq& \mathbb{E}[ \int_{0}^{T}\{ H(t,e_i) - H^{*}(t,e_i) \}\mathrm{d}t] \\
		& - \mathbb{E}[\int_{0}^{T}(X(t) - X^{*}(t))\big\{f_{x}^{*}(t ,e_i) + b_{x}^{*}(t,e_i)p(t)+ \sigma_{x}^{*}(t,e_i)q(t)\\
		&
		+\sum\limits_{j=1}^{D}\gamma^{j,*}_{x}(t,e_i)s_j(t)\zeta_{ij}(t) +\big( \mathbb{E}[f_{y}^{*}(t,e_{i})] + \mathbb{E}[b_{y}^{*}(t,e_i)p(t)]\\
		&+\mathbb{E}[\sigma_{y}^{*}(t,e_i)q(t)] + \sum\limits_{j=1}^{D}\mathbb{E}[\gamma^{j,*}_{y}(t,e_i)s_j(t)\zeta_{ij}(t)]\big)\varphi_{x}(X(t))\big\}\mathrm{d}t]\\
		&+ \mathbb{E}[\int_{0}^{T} p(t)^{\top}\{ G(t, e_i)\mathrm{d}(\xi(t) - \xi^{*}(t))\}\mathrm{d}t] + \mathbb{E}[\int_{0}^{T}\kappa(t)\mathrm{d}(\xi(t) - \xi^{*}(t))]\\
		=& \mathbb{E}[\int_{0}^{T}\{ H(t,e_i) - H^{*}(t,e_i)\}\mathrm{d}t] - \mathbb{E}[\int_{0}^{T} (X(t) - X^{*}(t))H_{x}(t,e_i) \mathrm{d}t]\\
		&- \mathbb{E}[ \int_{0}^{T}\mathbb{E}[H_{y}(t,e_i)]\varphi_{x}(X(t))
		(X(t) - X^{*}(t))\mathrm{d}t] \\
		&+ \mathbb{E}[ \int_{0}^{T}\{ \kappa(t) + G(t, e_i)p(t)\} \mathrm{d}(\xi(t) - \xi^{*}(t))].
	\end{align*}
By the concavity of the Hamiltonian in $u$, we have

\begin{align*}
	I_{1} + I_{2} + I_{3} \leq& \mathbb{E}[ \int_{0}^{T}\langle u(t) - u^{*}(t), \frac{\partial H^{*}(t,\alpha(t))}{\partial u} \rangle\mathrm{d}t] \\
	&+\mathbb{E}[ \int_{0}^{T}\{ \kappa(t) + G(t, e_i)p(t)\} \mathrm{d}( \xi(t) - \xi^{*}(t))].
\end{align*}
Since $u$ satisfies \eqref{eqn:man6maximumcondition}, we have

\begin{equation*}
	\mathbb{E}[\int_{0}^{T}\langle u(t) - u^{*}(t), \frac{\partial H^{*}(t,e_i)}{\partial u} \rangle\mathrm{d}t] \leq 0.
\end{equation*}
On the other hand, since $\xi^{*}$ satisfies \eqref{probcond}, let $\xi$ be an $\{\mathcal{F}_{t}\}_{t\in[0,T]}$-adapted c\`agl\`ad non-decreasing process with $\xi_{0} = 0$ and such that $\mathbb{P}( \int_{0}^{T} | G(t, e_i) |\mathrm{d}\xi(t) < \infty ) = 1$. 
Then, 
	\begin{align*}
		&\mathbb{E}[\int_{0}^{T}\{ \kappa(t) + G(t, e_i)p(t)\} \mathrm{d}(\xi(t) - \xi^{*}(t))] \\
		=& \mathbb{E}[ \int_{0}^{T}\big\{ \kappa(t) + G(t, e_i)\big\}p(t) )\mathrm{d}\xi(t)] \\
		&+ \mathbb{E}[ \int_{0}^{T}\mathbbm{1}_{\{\kappa(t) + G(t, \alpha(t-))p(t) < 0\}}\big\{ \kappa(t) - G(t, e_i)p(t) \big\}\mathrm{d}(-\xi^{*}(t))]\\
		&+ \mathbb{E}[ \int_{0}^{T}\mathbbm{1}_{\{\kappa(t) + G(t, e_i)p(t) = 0\}}\big\{ \kappa(t) - G(t, e_i)p(t) \big\}\mathrm{d}(-\xi^{*}(t))]\\
		=& \mathbb{E}[ \int_{0}^{T}\big(\kappa(t) - G(t,e_i)p(t)\big)\mathrm{d}\xi^{*}(t)] \leq 0.
	\end{align*}
\end{proof}

%\section{Applications}

\section{Application: Inter-bank borrowing and lending with transaction cost}\label{sec:application}
We consider a modification of the model of inter-bank borrowing and lending proposed in \cite{carmona2013mean}, where a number of banks borrow from and lend to each other based on the level of their reserves compared to a threshold set by a central bank. We assume that borrowing or lending comes with an additional cost, which we call the transaction cost. We denote the reserve of a representative bank at any given time $t$ by $X(t)$, and the threshold by $\mathbb{E}[X(t)]$. The representative bank's reserve is given by the following regime-switching mean-field dynamics:

\begin{align}
	\label{eqn:wealthapplication}
	\mathrm{d}X(t) =& \Big[ a(t,\alpha(t))\big\{\mathbb{E}[X(t)] - X(t)\big\} + b(t,\alpha(t))u(t) \Big]\mathrm{d}t + \sigma\mathrm{d}B(t)
	% + X(t)\sum\limits_{j=1}^{2}e_j(t)\mathrm{d}\tilde{\Phi}_j(t) 
	- c(t, \alpha(t))\mathrm{d}\xi(t),\notag\\
	X(0) =& x_0 > 0,
\end{align}
where $a$ is the rate of mean-reversion to the threshold, $b$ is an additional coefficient showing the dependence on the control on the state of the economy, $u(t)$ is the regular control representing the lending or borrowing rate (the amount to borrow or lend depending on whether the reserve $X(t)$ is below or above the threshold $\mathbb{E}[X(t)]$). Here, we assume that there are two regimes, the borrowing and lending regimes, which is modeled via a continuous time Markov chain $\alpha(t)$ defined on a finite state space $\mathcal{S}= \{e_1, e_2\}$.
%The coefficient of regime switching is given by $\alpha(t)X(t)$, which depends on the current level of the reserve $X(t)$, and  
The transaction cost $\xi(t)$ enters the model as an additional term with coefficient given as a function $c$ of time and the regime $\alpha(t)$. The transaction cost $\xi$ is a right continuous with left limits non-decreasing process of finite variation representing the singular control.

The representative bank's objective is to control its borrowing and lending rate at time $t$ and regime $\alpha(t)$ by selecting the rate $u(t)$ and the transaction cost $\xi(t)$ that minimises the cost functional

\begin{align}\label{valfuncapp0}
	J(u,\xi) =& \mathbb{E}\Big[ \int_{0}^{T}\big\{ \frac{1}{2}u^{2}(t) - \varrho u(t)\big(\mathbb{E}[X(t)] - X(t)\big) + \frac{\varepsilon}{2}(\mathbb{E}[X(t)] - X(t))^{2} \big\}\mathrm{d}t \notag\\
	&+\frac{\beta}{2}(\mathbb{E}[X(T)] - X(T))^{2} + \int_0^{T}\kappa(t)\mathrm{d}\xi(t) \Big],
\end{align} 
where $\varrho > 0$ controls the incentive to lend or borrow owing to the fact that a bank would like to borrow, i.e., $u(t)>0$ if its reserve $X(t)$ is smaller than the threshold $\mathbb{E}[X(t)]$ and vice versa. The quadratic terms in the running ($\varepsilon > 0$) and terminal ($\beta > 0$) costs penalize departure from the average. In order to ensure that running cost is convex in $(x, u)$, we assume (see, for example, \cite{carmona2013mean}) that $$\varrho^{2}\le \varepsilon.$$ 

In the sequel, we change the above minimisation problem to a maximisation to follow our earlier framework. In this case we state the optimal control problem as: Find the control $(u,\xi)\in\mathcal{U}$ such that 

\begin{align}
	\label{eqn:applicationobjective}
	J(u^{*},\xi^{*}) =& \max\limits_{(u,\xi)\in\mathcal{U}}\mathbb{E}\Big[ \int_{0}^{T}\big\{ -\frac{1}{2}u^{2}(t) + \varrho u(t)\big(\mathbb{E}[X(t)] - X(t)\big) - \frac{\varepsilon}{2}(\mathbb{E}[X(t)] - X(t))^{2} \big\}\mathrm{d}t \notag\\
	&- \frac{\beta}{2}(\mathbb{E}[X(T)] - X(T))^{2} - \int_0^{T}\kappa(t)\mathrm{d}\xi(t) \Big].
\end{align}
We use the sufficient maximum principle to solve the above problem. The Hamiltonian \eqref{eqn:man6hamiltonianmeanfieldsingular} becomes

	\begin{align}
		\label{eqn:applicationHamiltonian}
		H(t,x,y,u,\alpha, p)= -\frac{1}{2}u^{2} + \varrho u\big(y - x\big) - \frac{\varepsilon}{2}(y - x)^{2} + \big(a(t,\alpha)(y - x) + b(t,\alpha)u\big)p ,
	\end{align}
 and the associated first order adjoint equation \eqref{eqn:man6adjointforxmeanfieldsingular} becomes:

	\begin{align}
		\label{eqn:adjointapplication}
		\mathrm{d}p(t) =& -\big\{-\varrho u(t) + \varepsilon (\mathbb{E}[X(t)] - X(t)) - a(t,\alpha(t)) p(t) +\mathbb{E}\big[\varrho u(t)\big] \\
		&+ \mathbb{E}\big[a(t,\alpha(t))p(t)\big]\big\}\mathrm{d}t %+p(t)\mathrm{d}\xi(t) 
		+ q(t)\mathrm{d}B(t)+\sum\limits_{j=1}^{D}s_j(t)\mathrm{d}\tilde{\Phi}_j(t)
		,\notag\\
		p(T) =&\ \beta(\mathbb{E}[X(T)] - X(T)). 
	\end{align} 

The above linear mean-field BSDE can be rewritten as
\begin{align}\label{eq:BSDE-p}
p(t) =& -\beta\big(X(T)-\mathbb{E}[X(T)]\big)-\int_t^T
	\Big(
	\varrho\big(u(r)-\mathbb{E}[u(r)]\big)
\notag\\ 
&+ \varepsilon\big(X(r)-\mathbb{E}[X(r)]\big)	+ a(r,\alpha(r))\,p(r)-\mathbb{E}\big[a(r,\alpha(r))p(r)\big]
\Big)\,\mathrm{d}r\notag\\
&	+ \int_t^Tq(r)\,\mathrm{d}B(t)+\sum\limits_{j=1}^{D}\int_t^Ts_j(r)\mathrm{d}\tilde{\Phi}_j(r)
\end{align}
Taking the expectation on both sides of  \eqref{eq:BSDE-p} gives

\begin{align}\label{eq:BSDE-p1}
\mathbb{E}[	p(t)] =& -\beta\big(\mathbb{E}[X(T)]-\mathbb{E}[X(T)]\big)-\int_t^T
	\Big(
	\varrho\big(\mathbb{E}[u(s)]-\mathbb{E}[u(s)]\big)
	\notag\\ 
	&+ \varepsilon\big(\mathbb{E}[X(s)]-\mathbb{E}[X(s)]\big)	+ \mathbb{E}[a(s,\alpha(s))\,p(s)]-\mathbb{E}\big[a(t,\alpha(t))p(t)\big]
	\Big)\,\mathrm{d}s	\\
	&+ \mathbb{E}[\int_t^Tq(s)\,\mathrm{d}B(t)]+\sum\limits_{j=1}^{D}\mathbb{E}[\int_t^Ts_j(r)\mathrm{d}\tilde{\Phi}_j(r)]=0.
\end{align}

Therefore,
\begin{align}\label{zeroexpp}
\mathbb{E}[p(t)] = 0 \quad\text{for all } t\in[0,T].
\end{align}

To solve \eqref{eqn:adjointapplication}, we first consider the following linear SDE
	\begin{align}
		\label{eqn:adjointoflbsde}
		\mathrm{d}G_t^r =& G_t^{r} \big(-a(r,\alpha(r))\mathrm{d}r\big)
		% - \mathrm{d}\xi(r)\big)
		%+ \sum\limits_{j=1}^{2}e_{j}(r)\mathrm{d}\tilde{\Phi}_j(r)
		,\ G_t^{t} = 1,\ t \le r \le T.
	\end{align}

Using It\^o's product rule, we have

	\begin{align*}
		\mathrm{d}(p(r)G_t^{r})% =& p(r)\mathrm{d}G_t^{r} + G_t^{r}\mathrm{d}p(r) + \mathrm{d}[p(r), G_t^{r}]\\
		=&\ p(r)G_t^{r} \big(-a(r,\alpha(r))\mathrm{d}r \big)\\
		% + \sum\limits_{j=1}^{2}e_{j}(r)\mathrm{d}\tilde{\Phi}_j(r)\\
		&+ G_t^{r}\big[-\big\{-\varrho u(r) + \varepsilon (\mathbb{E}[X(r)] - X(r)) - a(r,\alpha(r))p(r) \notag\\
		&
		+\mathbb{E}\big[\varrho u(r)\big] + \mathbb{E}\big[a(r,\alpha(r))p(r)\big]\big\}\mathrm{d}r  + q(r)\mathrm{d}B(r)\big] \\
		=&\ \big\{ G_t^r \big(\varrho(u(r) - \mathbb{E}[u(r)])-\varepsilon(\mathbb{E}[X(r)] - X(r))- \mathbb{E}[a(r,\alpha(r))p(r)]\big)\big\}\mathrm{d}r\\
		& + G_t^{r}q(r)\mathrm{d}B(r)+\sum\limits_{j=1}^{D}G_t^{r}s_j(t)\mathrm{d}\tilde{\Phi}_j(t).
	\end{align*}
%%where $M$ %$$\mathrm{d}M(r) = G_t^r\Big(q(r)\mathrm{d}B(r) + \sum\limits_{j=1}^2\Big\{ e_j(r)p(r) + ({\bf{1}} + e_j(r))s_j(r) \Big\}\mathrm{d}\tilde{\Phi}_j(r)\Big)$$is a martingale. 
Integrating both sides from $t$ to $T$, taking the conditional expectation and using the fact that $G_t^{t} = 1$ give

\begin{align}\label{eqadjapp1}
	p(t) =& \mathbb{E}\Big[ p(T)G_t^{T} - \int_{t}^{T}\Big\{G_t^r \big(\varrho(u(r) - \mathbb{E}[u(r)])-\varepsilon(\mathbb{E}[X(r)] - X(r))\notag\\
	&- \mathbb{E}[a(r,\alpha(r))p(r)]\big)\Big\}\mathrm{d}r \mid \mathcal{F}_t  \Big]\notag\\
	=& \mathbb{E}_t\big[ p(T)G_t^{T}\big] - \int_{t}^{T}\Big\{\mathbb{E}_t\big[ G_t^r \big(\varrho(u(r) - \mathbb{E}[u(r)])-\varepsilon(\mathbb{E}[X(r)] - X(r))\big)\big]\notag\\
	&- \mathbb{E}[a(r,\alpha(r))p(r)]\mathbb{E}_t\big[ G_t^r\big]\Big\}\mathrm{d}r,
\end{align}
where $\mathbb{E}_{t}[\cdot]$ is the conditional expectation with respect to the filtration $\mathcal{F}_t$.
%Multiplying both sides of \eqref{eqadjapp1} by $a(t,\alpha(t))$ and using the fact that $a(t,\alpha(t))$ is $\mathcal{F}_t$-measurable
%\begin{align}\label{eqadjapp2}
%	&	a(t,\alpha(t))p(t)\notag\\
%	=& \mathbb{E}_t\big[ a(t,\alpha(t)) p(T)G_t^{T}\big] - \int_{t}^{T}\Big\{\mathbb{E}_t\big[ a(t,\alpha(t))G_t^r \big(\varrho(u(r) - \mathbb{E}[u(r)])\notag\\
%	&-\varepsilon(\mathbb{E}[X(r)] - X(r))\big)\big]- \mathbb{E}[a(r,\alpha(r))p(r)]\mathbb{E}_t\big[a(t,\alpha(t)) G_t^r\big]\Big\}\mathrm{d}r.
%\end{align}
%
Multiplying both sides of \eqref{eqadjapp1} by $a(t,\alpha(t))$ and using the fact that $a(t,\alpha(t))$ is $\mathcal{F}_t$-measurable

\begin{align}\label{eqadjapp2}
	&	a(t,\alpha(t))p(t)\notag\\
	=& \mathbb{E}_t\big[ a(t,\alpha(t)) p(T)G_t^{T}\big] - \int_{t}^{T}\Big\{\mathbb{E}_t\big[ a(r,\alpha(r))G_t^r \big(\varrho(u(r) - \mathbb{E}[u(r)])\notag\\
	&-\varepsilon(\mathbb{E}[X(r)] - X(r))\big)\big]- \mathbb{E}[a(r,\alpha(r))p(r)]\mathbb{E}_t\big[a(r,\alpha(r)) G_t^r\big]\Big\}\mathrm{d}r.
\end{align}
Taking the expectation on both sides of \eqref{eqadjapp2} yields

\begin{align*}
	%\label{eqadjapp3}
	&	\mathbb{E}\big[	a(t,\alpha(t))p(t)\big] \notag\\
	=& \mathbb{E}\big[ a(t,\alpha(t)) p(T)G_t^{T}\big] - \int_{t}^{T}\Big\{\mathbb{E}\big[ a(r,\alpha(r))G_t^r \big(\varrho(u(r) - \mathbb{E}[u(r)])\notag\\
	&-\varepsilon(\mathbb{E}[X(r)] - X(r))\big)\big]- \mathbb{E}[a(r,\alpha(r))p(r)]\mathbb{E}\big[a(r,\alpha(r)) G_t^r\big]\Big\}\mathrm{d}r\notag\\
	=& \mathbb{E}\big[ a(t,\alpha(t)) \beta(\mathbb{E}[X(T)] - X(T))G_t^{T}\big]+ \int_{t}^{T}\Big\{\varrho Cov(a(r,\alpha(r)) G_t^r,u(r))\notag\\
	& -\varepsilon  Cov(a(r,\alpha(r)) G_t^r,X(r))- \mathbb{E}[a(r,\alpha(r))p(r)]\mathbb{E}\big[a(r,\alpha(r)) G_t^r\big]\Big\}\mathrm{d}r\notag\\
	=& -\beta Cov( a(t,\alpha(t))G_t^T,X(T))+ \int_{t}^{T}\Big\{\varrho Cov(a(r,\alpha(r)) G_t^r,u(r))\notag\\
	& -\varepsilon  Cov(a(r,\alpha(r)) G_t^r,X(r))- \mathbb{E}[a(r,\alpha(r))p(r)]\mathbb{E}\big[a(r,\alpha(r)) G_t^r\big]\Big\}\mathrm{d}r.
\end{align*}
The above is a Volterra type equation and substituting its solution into \eqref{eqadjapp1} gives the solution to the adjoint equation. 

Using the maximum condition on the Hamiltonian \eqref{eqn:applicationHamiltonian}, it follows that the optimal control $u^*$ satisfies the following: 

\begin{align}
	\label{eqoptcontapp1}
	u^*(t) =& b(t,\alpha(t))p(t) + \varrho(\mathbb{E}[X(t)] - X(t)). 
\end{align}
Given the above solution to the BSDE \eqref{eqn:adjointapplication}, the singular control satisfies the following:
	\begin{align}\label{eqoptconsump1}
		\mathbb{P}\{\forall t\in [0,T], (-\kappa(t) - c(t, \alpha(t))p(t)) \le 0 \} = 1,\ \text{ and} \notag\\
		\mathbb{P}( \mathbbm{1}_{\{-\kappa(t) - c(t,\alpha(t))p(t) \le 0\}}\mathrm{d}\xi^{*}(t) = 0 ) = 1\ \text{a.s}. 
	\end{align}

The above findings can be summarise in the following Corollary: 

\begin{corollary}
	Suppose the dynamic of bank's reserve is given by \eqref{eqn:wealthapplication} and the goal of the bank is to find the optimal interest rate, and  consumption so as to minimise the cost functional \eqref{valfuncapp0}. The optimal interest rate and consumption rate are respectively given by \eqref{eqoptcontapp1} and \eqref{eqoptconsump1}. The adjoint process given by \eqref{eqadjapp1}.
\end{corollary}

\begin{remark}
	It is worth mentioning that since the coefficient $a$ in \eqref{eqn:wealthapplication} depends on the Markov chain $\alpha$, the mean-field BSDE \eqref{eqn:adjointapplication} cannot be solved explicitly as in \citep{agram2022mean}. 
		%Note that in \cite{wu2024maximum}, the problem was conditioned on the filtration generated by the Markov chain and thus $a$ is known and as a result the mean-field singular BSDE could be solved explicitly. 
		However, when $a$ is deterministic and does not depend on $\alpha$, the solution to the mean-field BSDE \eqref{eqn:adjointapplication} can be obtained explicitly. In fact, $G$ satisfies the following linear singular equation		
		\begin{align}
			\label{eqn:adjointoflbsden}
			\mathrm{d}G_t^r =& G_t^{r} \big(-a(r)\mathrm{d}r \big)
			%+ \sum\limits_{j=1}^{2}e_{j}(r)\mathrm{d}\tilde{\Phi}_j(r)
			,\ G_t^{t} = 1,\ t \le r \le T.
		\end{align}
	Using It\^o's product rule once more and \eqref{zeroexpp}, we have
	\begin{align*}
		\mathrm{d}(p(r)G_t^{r})
		=& \big\{ G_t^r \big(\varrho(u(r) - \mathbb{E}[u(r)])-\varepsilon(\mathbb{E}[X(r)] - X(r))\big)\big\}\mathrm{d}r+ G_t^{r}q(r)\mathrm{d}B(r)+\sum\limits_{j=1}^{D}s_j(t)\mathrm{d}\tilde{\Phi}_j(t).
	\end{align*}
	%%where $M$ %$$\mathrm{d}M(r) = G_t^r\Big(q(r)\mathrm{d}B(r) + \sum\limits_{j=1}^2\Big\{ e_j(r)p(r) + ({\bf{1}} + e_j(r))s_j(r) \Big\}\mathrm{d}\tilde{\Phi}_j(r)\Big)$$is a martingale. 
	Integrating both sides from $t$ to $T$, taking the conditional expectation and using the fact that $G_t^{t} = 1$ give
	\begin{align}\label{eqadjapp1n}
		p(t) =& \mathbb{E}_t\big[ p(T)G_t^{T}\big] - \int_{t}^{T}\Big\{\mathbb{E}_t\big[ G_t^r \big(\varrho(u(r) - \mathbb{E}[u(r)])-\varepsilon(\mathbb{E}[X(r)] - X(r))\big)\big]\Big\}\mathrm{d}r\notag\\
		=&\beta G_t^{T}\mathbb{E}_t\big[(\mathbb{E}[X(T)] - X(T))\big]\notag\\
		&- \int_{t}^{T}G_t^r\Big\{\mathbb{E}_t\big[  \big(\varrho(u(r) - \mathbb{E}[u(r)])-\varepsilon(\mathbb{E}[X(r)] - X(r))\big)\big]\Big\}\mathrm{d}r.
	\end{align}

\end{remark}

\section{Conclusion}
\label{sec:man6conclusion}
In this work, we have established necessary and sufficient conditions of optimality via the stochastic maximum principle for a model with regime switching and mean-field components. In deriving the necessary condition of optimality, we assumed that the regular control space was not convex, and so the spike variation approach was used whereas a convex perturbation was used for the singular control. The main task was the derivation of the variational inequality which involved second order adjoint processes since the regular control appears in the diffusion coefficient. In order to derive the sufficient maximum principle, we assumed that the regular control space was convex. Using the maximum condition, and the concavity of the Hamiltonian with respect to the regular control, the proof was done. As an example of situation where the current model can be applied, we considered a modified version an inter-bank borrowing and lending problem (see for example, \cite{carmona2013mean}) with transaction cost. Applying the sufficient maximum principle, we explicitly obtained the optimal rate of borrowing or lending for a representative bank. 

\appendix
	
\section{Proof of the auxiliary results}
\label{sec:proofofauxi}
Here we prove the auxiliary results.	
\begin{proof}[Proof of Lemma \ref{lem:man6lemma1}]
	From the controlled state process dynamics \eqref{eqn:man6meanfieldcontrolledstate}, we have 
	
		\begin{align*}
			%\begin{split}
			&|X^{\theta,\xi^\theta}(t)-X^{\theta,\xi^{*}}(t)|^2 \\ 
			=& |\int_0^t\{ b(s,X^{\theta,\xi^\theta}(s),\mathbb{E}[\varphi(X^{\theta,\xi^\theta}(s))], u^\theta(s),\alpha(s)) \\ 
			&- b(s,X^{\theta,\xi^{*}}(s),\mathbb{E}[\varphi(X^{\theta,\xi^{*}}(s))], u^\theta(s),\alpha(s)) \}\mathrm{d}s \\ 
			&+ \int_0^t\{ \sigma(s,X^{\theta,\xi^\theta}(s),\mathbb{E}[\varphi(X^{\theta,\xi^\theta}(s))],u^\theta(s),\alpha(s)) \\ 
			&- \sigma(s,X^{\theta,\xi^{*}}(s),\mathbb{E}[\varphi(X^{\theta,\xi^{*}}(s))],u^\theta(s),\alpha(s)) \}\mathrm{d}B(s) \\
			&+ \int_0^t\{ \gamma(s,X^{\theta,\xi^\theta}(s),\mathbb{E}[\varphi(X^{\theta,\xi^\theta}(s))],u^\theta(s),\alpha(s)) \\ 
			&- \gamma(s,X^{\theta,\xi^{*}}(s),\mathbb{E}[\varphi(X^{\theta,\xi^{*}}(s))],u^\theta(s),\alpha(s)) \}\mathrm{d}\tilde{\Phi}(s)  + \theta\int_{0}^{t} G(s, \alpha(s-))\mathrm{d}(\xi^\theta - \xi^{*})(s)|^2,
			%& \le C\sum_{i=1}^4 I^2_i(t).
			%\end{split}
		\end{align*}
	Taking the supremum on both sides, the expectation, and using the Lipschitz continuity of the coefficients, and applying Burkh\"older-Davis-Gundy (B-D-G) inequality, we obtain
		\begin{align*}
			%\begin{split}
			&	\mathbb{E}[\sup_{t\in[0,T]}|X^{\theta,\xi^\theta}(t)-X^{\theta,\xi^{*}}(t)|^2] \\
			\le& C \int_0^T \mathbb{E}[\sup_{t\in[0,s]}|X^{\theta,\xi^\theta}(t)-X^{\theta,\xi^{*}}(t)|^2]\mathrm{d}t \\
			&+ C \int_0^T\sum\limits_{j=1}^{D}\mathbb{E}[\sup_{t\in[0,s]}|X^{\theta,\xi^\theta}(t)-X^{\theta,\xi^{*}}(t)|^2]\zeta_{j}(t)\mathrm{d}t + C\theta^{2}  \|\int_0^T|\diffns \xi^*(t)|+|\diffns \iota(t)|\|_{L^{\infty}} ,
			%\end{split}
		\end{align*}

where the last term follows from the boundedness of $G$. Applying Gronwall's lemma yields
	
	\begin{equation}
		\label{eqn:man64.1}
		\mathbb{E}[\sup_{t\in[0,T]}|X^{\theta,\xi^\theta}(t)-X^{\theta,\xi^{*}}(t)|^2] \le C\theta^{2}.
	\end{equation}
	Therefore
	\begin{equation}
		\label{eqn:man64.11}
		\lim_{\theta \rightarrow 0}\mathbb{E}[\sup_{t\in[0,T]}|X^{\theta,\xi^*}(t)-X^{u^{*},\xi^{*}}(t)|^2] = 0.
	\end{equation}
	Using similar arguments as above, we have
	%\begin{equation}
	%\label{eqn:man64.111}
	$\mathbb{E}[\sup_{t\in[0,T]}|Z(t)|^p] < \infty.$ Let 
	\begin{equation}
		\label{eqn:man6Z1theta}
		Z_1^\theta(t) = \frac{X^{\theta,\xi^\theta}(t)-X^{\theta,\xi^{*}}(t)}{\theta} - Z(t), \text{ that is } X^{\theta, \xi^{\theta}}(t) - X^{\theta,\xi^{*}}(t) = \theta \big(Z_{1}^{\theta} + Z(t)\big)
	\end{equation}
	and set
	
	\begin{align*}
		b^{\mu^\theta}_x(t)=&b_x( X^{\theta,\xi^{*}}(t) + \mu(X^{\theta,\xi^\theta}(t) - X^{\theta,\xi^{*}}(t),\mathbb{E}[\varphi(X^{\theta,\xi^{*}})(t)] + \mu(\mathbb{E}[\varphi(X^{\theta,\xi^\theta})(t)] \\ 
		&- \mathbb{E}[\varphi(X^{\theta,\xi^{*}})(t)]),u^\theta(t),\alpha(t)),
	\end{align*} 
	%$$b^{\mu^\theta}_x(t)=b_x( X^{\theta,\xi^{*}}(t) + \mu(X^{\theta,\xi^\theta}(t) - X^{\theta,\xi^{*}}(t),\mathbb{E}[\varphi(X^{\theta,\xi^{*}})(t)] + \mu(\mathbb{E}[\varphi(X^{\theta,\xi^\theta})(t)] - \mathbb{E}[\varphi(X^{\theta,\xi^{*}})(t)]),u^\theta(t),\alpha(t)),$$
	and similarly for $b^{\mu^\theta}_y(t), \sigma^{\mu^\theta}_x(t), \sigma^{\mu^\theta}_y(t),\gamma^{\mu^\theta}_x(t)$, and  $\gamma^{\mu^\theta}_y(t)$.
	Also define	
	$$	b_x^*(t):=	b_x(t,X^{*}(t),\mathbb{E}[\varphi(X^{*}(t))],u^{*}(t),\alpha(t))$$ and similarly for other coefficients.
	Then, using mean-value theorem, we have
	
		\begin{align*}
			%\begin{split}
			&X^{\theta,\xi^\theta}(t)-X^{\theta,\xi^{*}}(t) \\
			&= \int_0^t\int_0^1b^{\mu^\theta}_x(t)(X^{\theta,\xi^\theta}(s)-X^{\theta,\xi^{*}}(s))\mathrm{d}\mu \mathrm{d}s + \int_0^t\int_0^1b^{\mu^\theta}_y(t)(\mathbb{E}[\varphi(X^{\theta,\xi^\theta})(s)] \\
			&-\mathbb{E}[\varphi(X^{\theta,\xi^{*}})(s)])\mathrm{d}\mu \mathrm{d}s + \int_0^t\int_0^1\sigma^{\mu^\theta}_x(s)(X^{\theta,\xi^\theta}(s)-X^{\theta,\xi^{*}}(s))\mathrm{d}\mu \mathrm{d}B(s) \\
			&+ \int_0^t\int_0^1\sigma^{\mu^\theta}_y(s)(\mathbb{E}[\varphi(X^{\theta,\xi^\theta})(s)]-\mathbb{E}[\varphi(X^{\theta,\xi^{*}})(s)])\mathrm{d}\mu \mathrm{d}B(s) \\
			&+ \int_0^t\int_0^1\gamma^{\mu^\theta}_x(s)(X^{\theta,\xi^\theta}(s)-X^{\theta,\xi^{*}}(s))\mathrm{d}\mu \mathrm{d}\tilde{\Phi}(s) \\
			&+ \int_0^t\int_0^1\gamma^{\mu^\theta}_y(s)(\mathbb{E}[\varphi(X^{\theta,\xi^\theta})(s)]-\mathbb{E}[\varphi(X^{\theta,\xi^{*}}(s))])\mathrm{d}\mu \mathrm{d}\tilde{\Phi}(s) + \theta\int_{0}^{t} G(s, \alpha(s-))\mathrm{d}(\iota - \xi^{*})(s).
			%					&+ \theta \int_0^t G(s)\mathrm{d}(\iota - \xi^{*})(s).
			%\end{split}
		\end{align*}
	
	Substituting the above into the expression of $Z_1^\theta$ we get
	
		\begin{align*}
			%\begin{split}
			&Z_1^\theta(t) \\
			=& \int_0^t\int_0^1 b^{\mu^\theta}_x(s)Z_1^\theta(s)\mathrm{d}\mu \mathrm{d}s +\int_0^t\int_0^1 \{b^{\mu^\theta}_x(s)-b^*_x(s)\}Z(s)\mathrm{d}\mu \mathrm{d}s \\
			&+ \int_0^t\int_0^1 b^{\mu^\theta}_y(s)\mathbb{E}[\int_0^1\varphi^{\vartheta^\theta}_x(s)\diffns \vartheta Z_1^\theta(s)]\mathrm{d}\mu \mathrm{d}s +\int_0^t\int_0^1 \{b^{\mu^\theta}_y(s)\mathbb{E}[\int_0^1\varphi^{\vartheta^\theta}_x(s)\diffns \vartheta Z(s)] \\
			%&+ \int_0^t\int_0^1 b^{\mu^\theta}_y(s)\mathbb{E}[\int_0^1\varphi^{\vartheta^\theta}_x(s)\diffns \vartheta Z_1^\theta(s)]\mathrm{d}\mu \mathrm{d}s \\
			& -b^*_y(s)\mathbb{E}[\varphi^*_x(X(s))Z(s)]\}\mathrm{d}\mu \mathrm{d}s + \int_0^t\int_0^1 \sigma_x^{\mu^\theta}(s)Z_1^\theta(s)\mathrm{d}\mu \mathrm{d}B(s) + \int_0^t\int_0^1 \{\sigma_x^{\mu^\theta}(s) \\
			&%&+ \int_0^t\int_0^1 \sigma_x^{\mu^\theta}(s)Z_1^\theta(s)\mathrm{d}\mu \mathrm{d}B(s) + \int_0^t\int_0^1 \{\sigma_x^{\mu^\theta}(s)-\sigma^*_x(s)\}Z(s)\mathrm{d}\mu \mathrm{d}B(s) \\
			-\sigma^*_x(s)\}Z(s)\mathrm{d}\mu \mathrm{d}B(s) + \int_0^t\int_0^1 \sigma^{\mu^\theta}_y(s)\mathbb{E}[\int^1_0\varphi^{\vartheta^\theta}_x(s)\diffns \vartheta Z_1^\theta(s)]\mathrm{d}\mu \mathrm{d}B(s) \\
			&+ \int_0^t\int_0^1 \{\sigma^{\mu^\theta}_y(s)\mathbb{E}[\int_0^1\varphi_x^{\vartheta^\theta}(s)\diffns \vartheta Z(s)] -\sigma^*_y(s)\mathbb{E}[\varphi^*_x(X(s))Z(s)]\}\mathrm{d}\mu \mathrm{d}B(s) \\
			&+ \int_0^t\int_0^1 \gamma^{\mu^\theta}_x(s)Z_1^\theta(s)\mathrm{d}\mu \mathrm{d}\tilde{\Phi}(s)
			%&+ \int_0^t\int_0^1 \gamma^{\mu^\theta}_x(s)Z_1^\theta(s)\mathrm{d}\mu \mathrm{d}\tilde{\Phi}(s)+ \int_0^t\int_0^1 \Big\{\gamma^{\mu^\theta}_x(s)-\sigma^*_x(s)\Big\}Z(s)\mathrm{d}\mu \mathrm{d}\tilde{\Phi}(s) \\
			+ \int_0^t\int_0^1 \{\gamma^{\mu^\theta}_x(s)-\sigma^*_x(s)\}Z(s)\mathrm{d}\mu \mathrm{d}\tilde{\Phi}(s) \\
			&+ \int_0^t\int_0^1 \gamma^{\mu^\theta}_y(s)\mathbb{E}[\int_0^1\varphi_x^{\vartheta^\theta}(s)\diffns \vartheta Z_1^\theta(s)]\mathrm{d}\mu \mathrm{d}\tilde{\Phi}(s)\\
			&+ \int_0^t\int_0^1 \{\gamma^{\mu^\theta}_y(s)\mathbb{E}[\int_0^1\varphi_x^{\vartheta^\theta}(s)\diffns \vartheta Z(s)] -\gamma^*_x(s)\mathbb{E}[\varphi^*_x(X(s))Z(s)]\}\mathrm{d}\mu \mathrm{d}\tilde{\Phi}(s). 
			%&+ \int_{0}^{t}\int_{0}^{1}G^{\mu^\theta}_{x}(s) Z_1(s)\mathrm{d}\mu\mathrm{d}\xi^*(s) + \int_{0}^{t}\int_{0}^{1}\{G^{\mu^\theta}_{x}(s) - G_{x}(s,X^{u^{*},\xi^{*}}(s))\}\\
			%%&+\int_{0}^{t}\int_{0}^{1}G^{\mu^\theta}_{x}(s) Z_1(s)\mathrm{d}\mu\mathrm{d}\xi^*(s) + \int_{0}^{t}\int_{0}^{1}\Big\{G^{\mu^\theta}_{x}(s) - G_{x}(s,X^{u^{*},\xi^{*}}(s)) \Big\}Z(s)\mathrm{d}\mu\mathrm{d}\xi^*(s)\\
			%& \times Z(s)\mathrm{d}\mu\mathrm{d}\xi^*(s) + \int_0^t \{G(s,X^{\theta,\xi^{*}})-G(s,X^{u^{*},\xi^{*}})\}\mathrm{d}(\iota-\xi^{*})(s).
		\end{align*}
	
	Therefore, squaring both sides, taking the supremum and the expectation, using B-D-G gives
	
	\begin{align*}
		\mathbb{E}&[\sup_{t\in [0,T]}|Z_{1}^{\theta}(t)|^{2}]\\
		\le& C\{\int_{0}^{T}\int_{0}^{1}(\mathbb{E}[\sup_{t\in [0,T]}|b_{x}^{\mu^\theta}(s)Z_{1}^{\theta}(s)|^{2}] + \mathbb{E}[\sup_{s\in [0,T]}|b^{\mu^\theta}_{y}(s) \mathbb{E}[\int_0^1\varphi^{\vartheta^\theta}_{x}(s)\diffns \vartheta Z_{1}^{\theta}(s)]|^{2}] \\
		&+ \mathbb{E}[\sup_{t\in [0,T]}|\sigma^{\mu^\theta}_{x}(s)Z_{1}^{\theta}(s)|^{2}] + \mathbb{E}[\sup_{s\in [0,T]}|\sigma^{\mu^\theta}_{y}(s)\mathbb{E}[\int_0^1\varphi^{\vartheta^\theta}_{x}(s)\diffns \vartheta Z_{1}^{\theta}(s)]|^{2}]\\
		&+ \mathbb{E}[\sup_{t\in [0,T]}|\gamma^{\mu^\theta}_{x}(s)Z_{1}^{\theta}(s)|^{2}]\zeta(s) + \mathbb{E}[\sup_{s\in [0,T]}|\gamma^{\mu^\theta}_{y}(s) \mathbb{E}[ \int_0^1\varphi^{\vartheta^\theta}_{x}(s)\diffns \vartheta Z_{1}^{\theta}(s)]|^{2}]\zeta(s) \\
		&+ \mathbb{E}[\sup_{s\in [0,T]}|(b^{\mu^\theta}_{x}(s) - b^*_{x}(s))Z(s)|^{2}] + \mathbb{E}[\sup_{s\in [0,T]}|(b_{y}^{\mu^\theta}(s) - b^*_{y}(s))Z(s)|^{2}]\\
		& + \mathbb{E}[\sup_{s\in [0,T]}| b^*_{y}(s) \mathbb{E}[ \int_0^1(\varphi^{\vartheta^\theta}_{x}(s)-\varphi_x(s))\diffns \vartheta Z(s)]|^{2}] \\
		&+ \mathbb{E}[\sup_{s\in [0,T]}|(\sigma^{\mu^\theta}_{x}(s) - \sigma^*_{x}(s))Z(s)|^{2}] + \mathbb{E}[\sup_{s\in [0,T]}|(\sigma_{y}^{\mu^\theta}(s) - \sigma^*_{y}(s))Z(s)|^{2}] \\
		&+ \mathbb{E}[\sup_{s\in [0,T]}|\sigma^*_{y}(s)  \mathbb{E}[ \int_0^1(\varphi^{\vartheta^\theta}_{x}(s)-\varphi_x(s))\diffns \vartheta Z(s)]|^{2}] \\
		&+ \mathbb{E}[\sup_{s\in [0,T]}|(\gamma^{\mu^\theta}_{x}(s) - \gamma^*_{x}(s))Z(s)|^{2}]\zeta(s) + \mathbb{E}[\sup_{s\in [0,T]}|(\gamma_{y}^{\mu^\theta}(s) - \gamma^*_{y}(s))Z(s)|^{2}]\zeta(s)\\
		%&+ \int_{0}^{T}\int_{0}^{1}\mathbb{E}\Big[\sup_{s\in [0,T]}\Big|\Big(\gamma_{y}^{\mu^\theta}(s) - \gamma^*_{y}(s)  \Big)Z(s)\Big|^{2}\Big]\zeta(s)\diffns \mu \mathrm{d}s\\
		&+ \mathbb{E}[\sup_{s\in [0,T]}| \gamma^*_{y}(s)  \mathbb{E}[ \int_0^1(\varphi^{\vartheta^\theta}_{x}(s)-\varphi_x(s))\diffns \vartheta Z(s)]|^{2}]\zeta(s)) \diffns \mu \mathrm{d}s\}.
		%			&\textcolor{red}{+(\mathbb{E}[\sup_{s\in [0,T]}|G^{\mu^\theta}_{x}(t) Z_1(t)] + \mathbb{E}[\sup_{s\in [0,T]}|\int_{0}^{1}\{G^{\mu^\theta}_{x}(s) - G_{x}(s,X^{u^{*},\xi^{*}}(s))\} Z(s)\mathrm{d}\mu|^2]} \\
		%			%&+\mathbb{E}\Big[\sup_{s\in [0,T]}\Big|G^{\mu^\theta}_{x}(t) Z_1(t)\Big]c_\infty	+ \mathbb{E}\Big[\sup_{s\in [0,T]}\Big|\int_{0}^{1}\Big\{G^{\mu^\theta}_{x}(s) - G_{x}(s,X^{u^{*},\xi^{*}}(s)) \Big\}Z(s)\mathrm{d}\mu\Big|^2\Big]c_\infty\\
		%			& \textcolor{red}{+ 2\mathbb{E}[\sup_{s\in [0,T]}|G(s,X^{\theta,\xi^{*}})-G(s,X^{u^{*},\xi^{*}})|^2])c_\infty\}}
		%	&\le C\int_{0}^{T}\mathbb{E}[|Z_{1}^{\theta}(t)|^{2}]\mathrm{d}t + C\int_{0}^{T}\sum\limits_{i=1}^{6}\left( \mathbb{E}[g_{i}(t,\psi_{1}(t)) - g_{i}(t)] \right)^{\frac{1}{2}}(\mathbb{E}[|Z(t)|])^{\frac{1}{2}}\mathrm{d}t + CM^{2}\|\xi(t)\|_{H^{\infty}}
	\end{align*}
Using both the continuity and the boundedness of $g_{i} = b_x, b_y, b_u, \sigma_x, \sigma_y,\sigma_u, \gamma_x, \gamma_y, \gamma_u$, and $\varphi_x$, the result follows by application of Gronwall's lemma.
\end{proof}

\begin{proof}[Proof of Lemma \ref{lem:man6lemma5}] 
	Let
		\begin{align*}
			\tilde{X}^{\theta,\xi^{*}}(t) = X^{\theta,\xi^{*}}(t) - \int_0^tG(s, \alpha(s-))\mathrm{d}\xi^{*}(s),\,
			\tilde{X}^{*}(t) = X^{*}(t) - \int_0^t G(s, \alpha(s-))\mathrm{d}\xi^{*}(s).
		\end{align*}

	Then 
	%\begin{equation*}
	$X^{\theta,\xi^{*}}(t) - X^{*}(t) - X_1(t) - X_2(t) = \tilde{X}^{\theta,\xi^{*}}(t) -  \tilde{X}^{*}(t) -  X_1(t) - X_2(t).$ 
	%\end{equation*}
	Similar to \cite{zhang2018general}, let $\bar{X}(t) = \tilde{X}^{\theta,\xi^{*}}(t) -  \tilde{X}^{*}(t) -  X_1(t) - X_2(t).$
	Then
		\begin{equation*}
			\mathrm{d}\bar{X}(t) = I_1^\theta(t;b)\mathrm{d}t + I_2^\theta(t;\sigma)\mathrm{d}B(t) + \sum\limits_{j=1}^{D} I_3^\theta(t;\gamma^{j})\mathrm{d}\tilde{\Phi}_{j}(t),
		\end{equation*}

	where

		\begin{align*}
			I_1^\theta(t;b) :=& b(t,X^{\theta,\xi^{*}}(t),\mathbb{E}[\varphi(X^{\theta,\xi^{*}}(t))],u^\theta(t),\alpha(t)) \\
			&- b(t,X^{*}(t),\mathbb{E}[\varphi(X^{*}(t))],u^\theta(t),\alpha(t)) - b^*_x(t)(X_1(t)+X_2(t))\\
			&- b^*_y(t)\mathbb{E}[\varphi_x(X(t))(X_1(t)+X_2(t))]- \frac{1}{2}b_{xx}(t)X_1^2(t)\\
			&- \delta b_x(t,u(t))\mathbbm{1}_{E_{\theta}}(t)X_1(t),\\
			I_{2}^{\theta}(t;\sigma) :=& \sigma(t,X^{\theta,\xi^{*}}(t),\mathbb{E}[\varphi(X^{\theta,\xi^{*}}(t))],u^\theta(t),\alpha(t)) \\
			&- \sigma(t,X^{*}(t),\mathbb{E}[\varphi(X^{*}(t))],u^\theta(t),\alpha(t)) - \sigma^*_x(t)(X_1(t)+X_2(t))\\
			&- \sigma^*_y(t)\mathbb{E}[\varphi_x(X(t))(X_1(t)+X_2(t))] - \frac{1}{2}\sigma_{xx}(t)X_1^2(t)\\
			&- \delta \sigma_x(t,u(t))\mathbbm{1}_{E_{\theta}}(t)X_1(t),\\
			I_{3}^{\theta}(t;\gamma^{j}) :=& \gamma^{j}(t,X^{\theta,\xi^{*}}(t),\mathbb{E}[\varphi(X^{\theta,\xi^{*}}(t))],u^\theta(t),\alpha(t)) \\
			&- \gamma^{j}(t,X^{*}(t),\mathbb{E}[X^{*}(t)],u^\theta(t),\alpha(t)) - \gamma^{j*}_x(t)(X_1(t)+X_2(t))\\
			&- \gamma^{j*}_y(t)\mathbb{E}[\varphi_x(X(t))(X_1(t)+X_2(t))]- \frac{1}{2}b_{xx}(t)X_1^2(t)\\
			&- \delta \gamma^{j}_x(t,u(t))\mathbbm{1}_{E_{\theta}}(t)X_1(t).
			%	I_{4}^{\theta}(t;G) :=& G(t,X^{\theta,\xi^{*}}(t)) - G(t,X^{*}(t)) - G_{x}(t,X^{*}(t))(X_{1}(t) + X_{2}(t)) \\
			%&- \frac{1}{2}G_{xx}(t,X^{*}(t))X_{1}^{2}(t).
		\end{align*}
	
	This term is analogous to \cite[Proposition 4.3]{zhang2018general}, and the proof follows by applying the same steps.

\end{proof}

\begin{proof}[Proof of Proposition \ref{prop:man6propositionpage201}]
	Similar to the proof of Lemma \ref{lem:man6lemma5}. 
\end{proof}

\bigskip
\begin{small}
		\smallskip\noindent\textbf{Funding~} 
		The work of M.~Ganet Some was carried out with the aid of a grant from the International Development Research Centre, Ottawa, Canada, \url{www.idrc.ca}, and with financial support from the Government of Canada, provided through Global Affairs Canada (GAC), \url{www.international.gc.ca};	Brandenburg University of Technology Cottbus-Senftenberg, Erasmus+ KA171 and research scholarship.
		He also gratefully acknowledges the  support by the German Academic Exchange Service, Grant/Award Number 57417894. \\
		E. Korveh was supported by the DAAD through the In-country/In-region PhD programme at AIMS Ghana.
		\\				
		%OM. Pamen was funded by the Alexander von Humboldt Foundation, under the programme financed by the German Federal Ministry of Education and Research entitled German Research Chair No. 01DG15010.
	
\end{small}

%\newpage

\bibliographystyle{acm}
\bibliography{references}

\begin{thebibliography}{10}

\bibitem{agram2022mean}
{\sc Agram, N., Hu, Y., and {\O}ksendal, B.}
\newblock Mean-field backward stochastic differential equations and
  applications.
\newblock {\em Systems \& Control Letters 162\/} (2022), 105196.

\bibitem{andersson2011maximum}
{\sc Andersson, D., and Djehiche, B.}
\newblock A maximum principle for {SDE}s of mean-field type.
\newblock {\em Appl. Math. Optim. 63\/} (2011), 341--356.

\bibitem{bahlali1996maximum}
{\sc Bahlali, K., Mezerdi, B., and Ouknine, Y.}
\newblock The maximum principle for optimal control of diffusions with
  non-smooth coefficients.
\newblock {\em Stochastics 57}, 3-4 (1996), 303--316.

\bibitem{bahlali2005general}
{\sc Bahlali, S., and Mezerdi, B.}
\newblock A general stochastic maximum principle for singular control problems.
\newblock {\em Electron. J. Probab. 10}, 30 (2005), 988--1004.

\bibitem{bensoussan1981lectures}
{\sc Bensoussan, A.}
\newblock Lectures on stochastic control, lecture notes in mathematics.
\newblock {\em Nonlinear filtering and stochastic control, proceedings, Cortona
  972\/} (1981).

\bibitem{buckdahn2009mean}
{\sc Buckdahn, R., Djehiche, B., Li, J., and Peng, S.}
\newblock Mean-field backward stochastic differential equations: a limit
  approach.
\newblock {\em Ann. Probab. 37}, 4 (2009), 1524--1565.

\bibitem{cadenillas1994stochastic}
{\sc Cadenillas, A., and Haussmann, U.~G.}
\newblock The stochastic maximum principle for a singular control problem.
\newblock {\em Stochastics: An International Journal of Probability and
  Stochastic Processes 49}, 3-4 (1994), 211--237.

\bibitem{cadenillas1995stochastic}
{\sc Cadenillas, A., and Karatzas, I.}
\newblock The stochastic maximum principle for linear, convex optimal control
  with random coefficients.
\newblock {\em SIAM J. Control Optim. 33}, 2 (1995), 590--624.

\bibitem{carmona2013mean}
{\sc Carmona, R., Fouque, J.-P., and Sun, L.-H.}
\newblock Mean field games and systemic risk.
\newblock {\em arXiv preprint arXiv:1308.2172\/} (2013).

\bibitem{dahl2017singular}
{\sc Dahl, K.~R., and {\O}ksendal, B.}
\newblock Singular recursive utility.
\newblock {\em Stochastics 89}, 6-7 (2017), 994--1014.

\bibitem{donnelly2011sufficient}
{\sc Donnelly, C.}
\newblock Sufficient stochastic maximum principle in a regime-switching
  diffusion model.
\newblock {\em Appl. Math. Optim. 64\/} (2011), 155--169.

\bibitem{elliott1990optimal}
{\sc Elliott, R.~J.}
\newblock The optimal control of diffusions.
\newblock {\em Appl. Math. Optim. 22}, 1 (1990), 229--240.

\bibitem{elliott1994exact}
{\sc Elliott, R.~J.}
\newblock Exact adaptive filters for {M}arkov chains observed in {G}aussian
  noise.
\newblock {\em Automatica J. IFAC 30}, 9 (1994), 1399--1408.

\bibitem{elliott2008hidden}
{\sc Elliott, R.~J., Aggoun, L., and Moore, J.~B.}
\newblock {\em Hidden Markov models: estimation and control}, vol.~29.
\newblock Springer Science \& Business Media, 2008.

\bibitem{hafayed2014singular}
{\sc Hafayed, M.}
\newblock Singular mean-field optimal control for forward-backward stochastic
  systems and applications to finance.
\newblock {\em Int. J. Dyn. Control 2\/} (2014), 542--554.

\bibitem{haussmann1986stochastic}
{\sc Haussmann, U.~G.}
\newblock {\em A stochastic maximum principle for optimal control of
  diffusions}.
\newblock John Wiley \& Sons, Inc., 1986.

\bibitem{hu2017singular}
{\sc Hu, Y., {\O}ksendal, B., and Sulem, A.}
\newblock Singular mean-field control games.
\newblock {\em Stoch. Anal. Appl. 35}, 5 (2017), 823--851.

\bibitem{kushner1972necessary}
{\sc Kushner, H.}
\newblock Necessary conditions for continuous parameter stochastic optimization
  problems.
\newblock {\em SIAM J. Control Optim. 10}, 3 (1972), 550--565.

\bibitem{kushner1965stochastic}
{\sc Kushner, H.~J.}
\newblock On the stochastic maximum principle: Fixed time of control.
\newblock {\em J. Math. Anal. Appl. 11\/} (1965), 78--92.

\bibitem{lasry2007mean}
{\sc Lasry, J.-M., and Lions, P.-L.}
\newblock Mean field games.
\newblock {\em Jpn. J. Math. 2}, 1 (2007), 229--260.

\bibitem{Menou20142}
{\sc Menoukeu-Pamen, O.}
\newblock Maximum principles of {M}arkov regime-switching forward-backward
  stochastic differential equations with jumps and partial information.
\newblock {\em J. Optim. Theory Appl. 175\/} (2017), 373--410.

\bibitem{MeMo17}
{\sc Menoukeu-Pamen, O., and Momeya, R.}
\newblock A maximum principle for markov regime-switching forward--backward
  stochastic differential games and applications.
\newblock {\em Math. Methods Oper. Res. 85\/} (2017), 349--388.

\bibitem{Thilo2012}
{\sc Meyer-Brandis, T., {\O}ksendal, B., and Zhou, X.~Y.}
\newblock A mean-field stochastic maximum principle via malliavin calculus.
\newblock {\em Stochastics, control and finance 84}, 5--6 (2012), 643--666.

\bibitem{nguyen2020stochastic}
{\sc Nguyen, S.~L., Nguyen, D.~T., and Yin, G.}
\newblock A stochastic maximum principle for switching diffusions using
  conditional mean-fields with applications to control problems.
\newblock {\em ESAIM Control Optim. Calc. Var. 26\/} (2020), 69.

\bibitem{protter2005stochastic}
{\sc Protter, P.~E.}
\newblock {\em Stochastic Integration and Differential Equations}, vol.~21.
\newblock Springer Science \& Business Media, 2005.

\bibitem{shen2013maximum}
{\sc Shen, Y., and Siu, T.~K.}
\newblock The maximum principle for a jump-diffusion mean-field model and its
  application to the mean--variance problem.
\newblock {\em Nonlinear Anal. 86\/} (2013), 58--73.

\bibitem{tao2012maximum}
{\sc Tao, R., and Wu, Z.}
\newblock Maximum principle for optimal control problems of forward--backward
  regime-switching system and applications.
\newblock {\em Systems Control Lett. 61}, 9 (2012), 911--917.

\bibitem{wu2024maximum}
{\sc Wu, Z., and Zhang, Y.}
\newblock Maximum principle for conditional mean-field {FBSDEs} systems with
  regime-switching involving impulse controls.
\newblock {\em Journal of Mathematical Analysis and Applications 530}, 2
  (2024), 127720.

\bibitem{zhang2012stochastic}
{\sc Zhang, X., Elliott, R.~J., and Siu, T.~K.}
\newblock A stochastic maximum principle for a markov regime-switching
  jump-diffusion model and its application to finance.
\newblock {\em SIAM J. Control Optim. 50}, 2 (2012), 964--990.

\bibitem{zhang2018general}
{\sc Zhang, X., Sun, Z., and Xiong, J.}
\newblock A general stochastic maximum principle for a markov regime switching
  jump-diffusion model of mean-field type.
\newblock {\em SIAM J. Control Optim. 56}, 4 (2018), 2563--2592.

\end{thebibliography}
\addcontentsline{toc}{chapter}{References}

\newpage
\footnotesize

\end{document}